\documentclass{amsart}
\usepackage{times}
\usepackage{amsthm}
\usepackage{amssymb}
\usepackage[pdftex]{graphicx}
\usepackage[all]{xy}
\usepackage[mathscr]{eucal}
\usepackage{amsmath,latexsym,oldlfont}
\usepackage{mathdots}
\usepackage{pifont}
\usepackage{stmaryrd}
\usepackage{textcomp}
\usepackage{pifont}
\usepackage{multirow}
\usepackage[colorlinks,linkcolor=blue,citecolor=red,linktocpage=true]{hyperref}
\usepackage{mathtools}
\numberwithin{equation}{section} 
\newtheorem{lem}{Lemma}
\newtheorem{thm}[lem]{Theorem}
\newtheorem{cor}[lem]{Corollary}
\newtheorem{prop}[lem]{Proposition}
\newtheorem{example}[lem]{Example}
\theoremstyle{definition}
\newtheorem{defn}{Definition}
\theoremstyle{remark}
\newtheorem{rem}{Remark}

\newcommand{\swedge}{{\scriptstyle\wedge}}

\newcommand{\nwedge}{\mathchoice{{\textstyle\wedge}}%
    {{\wedge}}%
    {{\textstyle\wedge}}%
    {{\scriptstyle\wedge}}}
\newsavebox{\spacebox}
\begin{lrbox}{\spacebox}
    \verb*! !
\end{lrbox}

\title[On   Lagrangian-Grassmannian Variety]{On   Lagrangian-Grassmannian Variety}
\author[J. Carrillo--Pacheco]{Jes\'us Carrillo--Pacheco}
\address[J. Carrillo--Pacheco]{Academia de Matem\'aticas, Universidad Aut\'onoma de la Ciudad de M\'exico, 09390 M\'exico, Ciudad de M\'exico.}

\email[J. Carrillo--Pacheco]{jesus.carrillo@uacm.edu.mx}

\begin{document}



\begin{abstract}
In this paper it is shown that Family of Linear Relations of Contraction ($FRLC$) are the only ones, up to linear combination, that vanish the Lagrangian-Grassmannian. It is shown that the Pl\"ucker matrix of the Lagrangian-Grassmannian is a direct sum of  incidence matrix, regular and sparce with entries in the set  \{ 0, 1\}. 
\end{abstract}
 \maketitle
\section{Introduction}\label{Intro}
In  words of {\it Dusa McDuff} 
"{\it Symplectic geometry is the geometry of closed skew-symmetric form, $[\cdots]$
thus symplectic geometry is essentially topological in nature}" $[\cdots]$, see \cite{bib 2.23}.\\
In this article it is shown that the Lagrangian-Grassmannian also has a geometrical-algebraic nature,
to show this we exhaustively study the set of all homogeneous linear functionals in the vector space $\wedge^nE$, which cancel $L(n, E)$. There are at least two ways to approach this geometrical-algebraic aspect of the Lagrangian-Grassmannian $L(n, E)$, one of them is a classical approach and the other one that we develop here is a slightly different approach.\\
\it{Classic mode):} is given by De Concini and Lakshmibai [1981] \cite{bib2.111} where it is shown that, in its natural projective embedding, the Lagrangian Grassmannian $L(n, E)$ is defined by quadratic relations. These relations are obtained by expressing $L(n, E)$ as a linear section of $G(n, E)$,
so $L(n,E)=G(n,E)\cap {\mathbb P}(L(\omega_n))$, where $ {\mathbb P}(L(\omega_n))$ is the projectivization of a vector space $L_{\omega_n}$ such that $\wedge^nE\simeq L_{\omega_n}\oplus \wedge^{n-2}E$, where $L_{\omega_n}$ is $S_{p_ {2n}}({\mathbb F})$-representation of highest weight $\omega_n=h_1^{*}+\cdots +h_n^{*}$, and where $L(n,E)\subset L_{ \omega_n}$ see \cite[pages 182-184]{bib 2.1112}. This approach is important as can be seen in \cite[section 5]{bib0.001}, \cite[section 4.3]{bib 2.12}, \cite[section 2.2]{bib 2.13}, \cite[p. 823]{bib5a} for example, where they are used to obtain different results about $L(n, E)$.\\ 
Another way, not classic, is to see $L(n, E)=G(n, E)\cap ker f$ where $f$ is the contraction map, in \cite[prop 6]{bib2} all linear functionals homogeneous that nullify the kernel of $f$ are  given. So then it is natural to ask what is the smallest family of all homogeneous linear functionals of $(\wedge^nE)^*$ that vanish $L(n, E)$. We call this collection \it{Family of Linear Relations of Contraction}, short $FRLC$, and we will show its existence and uniqueness. In this article we show that these functionals are, up to a linear combination, all those that cancel $\ker f$ and are the only ones with this property. See \cite[prop 2.1]{bib0.00001}, \cite[prop 2.3]{bib0.000001}, \cite[pag. 4]{bib 2.12334} and \cite[pag. 386]{bib 2.13}, where you can see the importance of these $FRLC$. \\
We call the matrix associated to $FRLC$, the Pl\"ucker Matrix of the Lagrangian-Grassmannian and it is the direct sum of $(0, 1)$-submatrices, regular and sparse. We will also show that these submatrices are the incidence matrices of a collection of subsets that we define here.\\
We can say that the importance of this article resides on the one hand in demonstrating that the kernel of the contraction map is the smallest linear projective space, which contains $L(n, E)$ for any field. And the second aspect is to provide the form of the P\"ucker matrix of the Lagrangian-Grassmannian, which depends on incidence sets.\\
The following aspects do not appear in this article, but they are mentioned since they are part of a work subsequent to this investigation: In \cite[example 11]{bib1.1} it was shown that the functionals that annul the kernel of contraction map  also vanish the Isotropic Grassmannian $IG(k, E)$ which parameterizes all $k$-isotropic vector subspaces of a simplectic vector space $E$, so then it remains to prove whether the family is the smallest that vanishes $IG( k, E)$ and determine its pl\"ucker matrix. The $(0,1)$-regular and sparse matrices that appear in section $4$ induce Parity Check Codes and $LDPC$ see \cite{bib 2.11123} , \cite{bib0.0012} and finally these matrices are quasi-similar which is a concept we introduce here.
\section{Preliminaries}\label{PreTer}
Using terminology and definitions as given in \cite{bib0.01} and \cite{bib3},
let $E$ be an vector space defined over an arbitrary field ${\mathbb F}$. A  {\it symplectic form} 
is a bilinear map
$$\langle\; ,\;\rangle:E\times  E \longrightarrow {\mathbb F}$$
 that satisfies 
\begin{alignat*}{2}
 \langle v , w\rangle& =-\langle w , v\rangle &\qquad for\;\; all\;\; v, w\in E \\
 \langle v , v\rangle& = 0 &\qquad for\;\; all\;\; v\in E \\
 and\;\;if\; \langle v , w\rangle & = 0   &\qquad for\;\; all\;\; v\in E \Rightarrow w=0\\ 
  \end{alignat*}
and it is said to be skew-symmetric nondegenerate. $(E, \langle, \rangle)$ is called a  {\it symplectic vector space},
a  symplectic vector space $E$ is necessarily of even dimension and there is a basis $e_1,\ldots, e_n, u_1,\ldots u_n$ of $E$ such that $$\langle e_i, e_j \rangle=\langle u_i, u_j \rangle=0 \;\; y \;\; \langle e_i, u_j \rangle=\delta_{ij} $$
 where $\delta_{ij}$ is the Kroneker delta function. 
 A  subspace $U\subseteq E$ is said to be isotropic if $\langle u, u^{\prime} \rangle=0$ for all $u, u^{\prime}\in U$. A subspace $L\subset E$ is said to be \it{Lagrangian subspace} if $L$ is isotropic and $\dim (L)=n$. The collection of Lagrangian subspace of $E$, we call it \it{Lagrangian-Grassmannian} of $E$ or simply Lagrangian-Grassmannian we denote it by $L(n, E)$. A subspace $W\subset E$ is a symplectic subspace of $E$ if the symplectic form in $E$ when restricted to $W$ remains symplectic,
  \subsubsection{indices} Let $m$ be an integer we denote by
  \begin{equation}\label{mbrak}
   [m]:=\{1, 2, \ldots, m \}
   \end{equation}
    to the set of the first $m$ integers.
Let $m$ and $\ell$ be positive integers such that  $\ell< m$ as usual in the literature 
$C^m_{\ell}$ denotes binomial coefficient. If  $\alpha=(\alpha_1,\ldots,\alpha_{\ell})\in {\mathbb N}^{\ell}$ then,  $supp\{ \alpha\}=\{ \alpha_1, \cdots, \alpha_{\ell}\}$, 
 If $s\geq 1$ is a positive integer and  $\Sigma\subset {\mathbb N}$ is a non-empty set,   
we define  the  sets
\begin{equation}\label{setcomb1}
C_s(\Sigma):=\{\alpha=(\alpha_1, \ldots, \alpha_{s})\in {\mathbb N}^s: \alpha_1< \ldots < \alpha_{s} \;and\; supp \{\alpha \}\subseteq \Sigma \}
\end{equation}
Clearly  if $|\Sigma|=m$ then $|C_s(\Sigma)|=C^m_{\ell}$, 
whit this notation if $\ell < m$  we define $I(\ell, m)=C_{\ell}([m])$. so we have
\begin{equation}\label{setcomb2}
I(\ell,m)=\{\alpha=(\alpha_1, \ldots, \alpha_{s}) :
1\leq\alpha_1<\cdots<\alpha_{\ell}\leq m \}
\end{equation}
In general we say 
\begin{equation}
\alpha=(\alpha_1, \ldots, \alpha_s)\in C_s(\Sigma)  
\end{equation}
it there is a permutation $\sigma$ such that arrange the elements of $supp \{\alpha\}$ in increasing order
we have $(\sigma(\alpha_1), \ldots, \sigma(\alpha_s))\in C_s(\Sigma)$. \\
Also if $\alpha$ and $\beta$ are elements of $C_s(\Sigma)$ then we say that 
\begin{equation}\label{indequal000}
\alpha=\beta\;  \Leftrightarrow \;  {\text supp}\{\alpha\}={\text supp}\{\beta\}
\end{equation} 
Let $\alpha \in I(n,2n)$, suppose there are $i, j \in supp\; \alpha$ such that $i+j=2n+1$ in this case 
$j=2n-i+1$ and we write this pair as $P_i= (i, 2n-i+1)$ so we define the set  
\begin{equation}\label{setSigma0001}
\Sigma_n=\big\{ P_1, \ldots, P_n\big\}
\end{equation}
 and if  $\alpha \in I(n,2n)$ such that $\{ i, 2n-i+1\}\subset supp\{ \alpha \}$, then we say that 
$P_i\in supp\{ \alpha \}$ and that $P_i\in supp\{\alpha\}\cap \Sigma_m$.\\
If $k\leq n$ is a positive integer we denote by
\begin{equation}
C_{k/2}(\Sigma_n):= \{P_{\beta}=(P_{\beta_1}, \ldots, P_{\beta_k}): \beta=(\beta_1, \ldots, \beta_k)\in I(k/2, n)) \}
\end{equation}
If $1\leq a_1<a_2<\cdots<a_{2k}\leq 2n$ such that $a_i+a_j\neq 2n+1$  then we define 
\begin{equation}\label{setSigma0011}
\Sigma_{a_1,\ldots,a_{2k}}=\Sigma_n-\{P_{a_1}, \dots, P_{a_{2k}}\}
\end{equation}
so $|\Sigma_{a_1,\ldots,a_{2k}}|=n-2k$. 
Clearly if $n\geq 4$ then for each $\alpha\in I(n, 2n)$ we have to $|supp \{\alpha\}\cap \Sigma_n|=0$ or $|supp \{\alpha\} \cap \Sigma_n|=|supp \{\alpha\}|$ or $1\leq |supp \{\alpha\}\cap \Sigma_n|\leq\lfloor\frac{n-2}{2}\rfloor$ respectively.
Note that $|supp P_{\beta}|=2|supp \beta|$.\\
Given a canonical basis  ${\mathcal B}=\{e_1,\ldots, e_n, u_1,\ldots u_n\}$, see \cite{bib0.01}, of simplectic vector space   $E$, in this article we redefine its elements as follows $e_{n+1}:=u_1,\ldots e_{2n}:=u_n$ and we have 
 ${\mathcal B}=\{e_1,\ldots,e_{2n}\}$ such that
\begin{equation}\label{sympbase1}
\langle e_i,e_j\rangle=\begin{cases}
1& \text{if $j=2n-i+1$}, \\
0 & \text{otherwise}.
\end{cases}
\end{equation}
With such a choice the symplectic form $\langle,\rangle$ can be described as follows,  let $x=(x_1, x_2, \ldots, x_{2n} )$ and $y=(y_1, y_2,\ldots, y_{2n})$ then $\langle x, y \rangle=\sum_{i=1}^n [(x_i\cdot y_{2n+i-1})-(x_{2n+1-i}\cdot y_i)]$. 
 It is easy to verify that pairing is a non-degenerate alternanting on the vector space $E$. We call the above form the {\it standard symplectic form}.
For $\alpha=(\alpha_1,\ldots,\alpha_n)\in I(n,2n)$ write
 \begin{align}\label{notbasis}
e_{\alpha}&=e_{\alpha_1}\swedge\cdots\swedge e_{\alpha_n},\\
e_{\alpha_{rs}}&=e_{\alpha_1}\swedge\cdots\swedge\widehat{e}_{\alpha_r}
\swedge\cdots\swedge\widehat{e}_{\alpha_s}\swedge\cdots   \swedge e_{\alpha_n},\\
p_{i,\alpha_{rs},(2n-i+1)}&=p_{i,\alpha_1 \cdots \widehat{\alpha}_r
\cdots \widehat{\alpha}_s \cdots, \alpha_n(2n-i+1)}
\end{align}
where $\widehat{e}_{{\alpha}_k}$ and $\widehat{\alpha}_k$ means that 
the corresponding term is omitted y $p_{i\alpha_{rs}(2n-i+1)}$ is a scalar. Denote  by $\nwedge^nE$ the
$n$-th exterior power of $E$, which is  generated by $\{ e_{\alpha}:
\alpha \in I(n,2n)\}$. For $w=\sum_{\alpha\in
I(n,2n)}p_{\alpha}e_{\alpha} \in \nwedge^nE$, the coefficients
$p_{\alpha}$ are the {\it Pl\"ucker coordinates of} $w$, see also \cite[pag42]{bib6}.
As in \cite[pg 283]{bib 2.122},
let $n\geq 2$ integer and $E$ simplectic vector space of dimension $2n$, to the linear transformation
\begin{gather*}
f:\nwedge^nE\rightarrow \nwedge^{n-2}E\\
f(w_1\swedge\cdots\swedge w_n)=\sum_{1\leq r<s\leq n}\langle
w_r,w_s\rangle (-1)^{r+s-1}w_1\swedge\cdots\swedge
\widehat{w}_r\swedge\cdots\swedge \widehat{w}_s\swedge\cdots\swedge
w_n,
\end{gather*}
where $\widehat{w}$ means that the corresponding term is omitted ver \cite[pag 283]{bib 2.122} we call him 
\textit { contraction map}.\\
%
For an $m$-dimensional vector space $E$, denote by $G(\ell, E)$ the set of vector subspaces of dimension $\ell$ of $E$. The Grassmannian $G(\ell, E)$ is a algebraic variety of dimension $\ell(m-\ell)$ and can be embedded in a projective space ${\mathbb P}^{c-1}$, where $c=\binom{m}{\ell}$ by Pl\"ucker embedding.
The {\it Pl\"ucker embedding}
is the injective mapping $\rho:G(\ell,E)\rightarrow {\mathbb
P}(\wedge^mE)$ given on each $W\in G(\ell, E)$ by choosing  a basis
$w_1,\ldots, w_{\ell}$ of $W$ and then mapping the vector subspace
$W\in G(\ell, E)$ to the tensor $w_1\swedge\cdots\swedge w_{\ell}\in
\nwedge^{\ell}E$. Since choosing a different basis  of $W$ changes the
tensor $w_1\swedge\cdots\swedge w_{\ell}$ by a nonzero scalar, this
tensor is a well-defined element in the projective space ${\mathbb
P}(\nwedge^{\ell}E)\simeq{\mathbb P}^{N-1}$, where $N=C^{m}_{\ell}=\dim_{\mathbb{F}}(\nwedge^{\ell}E)$. 
Writing 
 $w \in \wedge^{\ell} E$ as $w= \sum_{\alpha
\in I(\ell,m)} P_\alpha e_\alpha$, the scalars $P_\alpha$ are called  the {\it
Pl\"ucker coordinates} of $w$ and $w_\rho=(P_\alpha)_{\alpha \in I(\ell, m)}$ is
 the  {\it Pl\"ucker vector} of $w$. 
If $w = \sum_{\alpha \in I(\ell,m)}P_\alpha e_\alpha \in
{\mathbb{P}}(\wedge^{\ell} E)$, then $w\in G(\ell,m)$ if and only if for
each pair of tuples $1\leq \alpha_1< \cdots <
\alpha_{\ell-1}\leq m$ and $1\leq \beta_1< \cdots < \beta_{\ell+1}\leq m$, the Pl\"ucker coordinates of $w$ satisfy the quadratic \it{Pl\"ucker relation}
\begin{equation}\label{eq1.1}
 Q_{\alpha, \beta}:=\sum_{i=1}^{\ell+1} (-1)^i P_{\alpha_1\cdots \alpha_{\ell-1}\beta_i}P_{\beta_1, \beta_2\cdots
\widehat{\beta_i} \cdots \beta_{\ell+1}} = 0,
\end{equation}
where  $\widehat{\beta_i}$ means that the corresponding term is omitted and where $\alpha=(\alpha_1\cdots \alpha_{\ell-1}) \in I(\ell-1, m)$,  $\beta=(\beta_1, \beta_2\cdots
\beta_i \cdots \beta_{\ell+1}) \in I(\ell+1,m)$, see \cite[section 4]{bib6}. 
Under the inclusion of  Pl\"ucker the  Lagrangian-Grassmannian  is given by
\begin{gather}\label{LGinGrass}
L(n,E)=\{w_1\swedge\cdots\swedge w_n\in G(n,E): \langle w_i,w_j\rangle=0\;\text{for all $1\leq i<j\leq n$}\}
\end{gather}
Can easily be seen
\begin{gather}\label{LGint}
L(n, E)=G(n,E)\cap \ker f
\end{gather}
In \cite[Proposition 6]{bib2} the kernel of the contraction map $f$ is
characterized as follows: For $w=\sum_{\alpha\in
I(n,2n)}p_{\alpha}e_{\alpha}\in \nwedge^nE$ written in Pl\"ucker
coordinates,  we have that
$$w\in\ker f\iff \sum_{i=1}^np_{i\alpha_{st}(2n-i+1)}=0,\;\text{for all $\alpha_{st}\in I(n-2,2n)$}.$$
For all $\alpha_{rs}\in I(n-2,2n)$ define
\begin{gather}\label{FuncLinPI}
\Pi_{\alpha_{rs}}=\sum_{i=1}^nc_{i,\alpha_{rs}, 2n-i+1} X_{i,\alpha_{rs}, 2n-i+1} 
\end{gather}
where 
$$c_{(i,\alpha_{rs}, 2n-i+1)}=\begin{cases}
1  & \text{if $|\text{supp}\{(i, \alpha_{rs}, 2n-i+1)\}|=n$}, \\
0 & \text{otherwise},
\end{cases}$$
\begin{defn}\label{relinpluker}
Let $E$ symplectic vector space of dimension $2n$.\\
To the smallest family of all homogeneous linear functionals of $(\wedge^nE)^*$ that vanish $L(n, E)$, up to linear combination, we call   \it{Family of Linear Relations of Contraction} of $L(n, E)$, in short $FRLC$.
\end{defn}
Following  \cite{bib 2.1233} if $A\subset {\mathbb P}^s$ is an algebraic set then $Z\langle A \rangle$ denotes the set of zeros of $A$ in a projective space. Thus 
\begin{gather}\label{ZeroSet}
Z\langle \Pi_{\alpha_{rs}} : \alpha_{rs}\in I(n-2, 2n) \rangle\subseteq {\mathbb P}(\wedge^nE) 
\end{gather}
is the set of zeros all $\Pi_{\alpha_{rs}}\in(\wedge^nE)^*$.   
Using  \cite[lemma 2]{bib2} and the above notation we have
\begin{equation}\label{lagrzeroset}
L(n, E)=Z\langle Q_{\alpha, \beta}, \Pi_{\alpha_{rs}} \rangle
\end{equation}
where $Q_{\alpha, \beta}$ and $\Pi_{\alpha_{rs}} $ are as in \ref{eq1.1} and  \ref{FuncLinPI} respectively.
\begin{example}
In the case $L(2, E)$ we have the following equations
\begin{align}\label{lagrzeroset2} 
\begin{cases}
X_{12}X_{34}-X_{13}X_{24}+X_{14}X_{23}&=0\\
X_{14}+X_{23}&=0
\end{cases}
\end{align}
\end{example}
%
The rational points of $L(n,E)$ are defined as the set $$L(n,E)({\mathbb F}_q):=Z\langle Q_{\alpha^{\prime}, \beta^{\prime}}, \Pi_{\alpha_{rs}}, x^q_{\alpha}-x_{\alpha} \rangle$$
where $\alpha\in I(n-1, 2n)$,  $\beta\in I(n+1, 2n)$,  $\alpha_{rs}\in I(n-2,2n)$ and  $\alpha \in I(n, 2n)$
more over 
\begin{equation}\label{rationPoint1}
|L(n,E)({\mathbb F}_q)|=\Pi_{i=1}^{n}(1+q^i)
\end{equation}
see \cite[prop. 2.14]{ bib 2.131}.
\subsubsection{matrices}
If  a matrix $A$ has all its coefficients equal  $0$ or $1$  is called a $(0,1)$-{\it matrix}. Give a $(0,1)$-matrix $A$ we say that  is {\it regular} if the number of 1's is fixed in each column and 
has a fixed number of 1's in each row. If $A$ is not regular we say that  is {\it irregular} see \cite{bib5a} and 
\cite{bib5}, for more information. 
A  {\it sparse matrix} is a $(0,1)$-matrix in which most of the elements are zero.\\
\subsubsection{Configuration of subsets}
Following \cite[pag. 3]{bib0.0012} we call  $X=\{ x_1, \ldots, x_n\}$ an $n$-set. Now let $X_1, X_2, \dots, X_m$ be $m$ not necessarily distinct subsets of the $n$-set $X$. We refer to this  colletions of subsets of an $n$-set as a  {\it configuration of  subsets}. We set  $a_{ij}=1$ if $x_j\in X_i$  and we set $a_{ij}=0$ if $x_j\notin X_i$.  The resulting $(0, 1)$-matrix $A=(a_{i,j})$, $i=1,\ldots, m$, $j=1,\ldots, n$ of size $m$ by $n$ is the \it{incidense matrix}
 for the configurations of subsets $X_1, X_2, \dots, X_m$ of the $n$-set $X$. The $1^{\prime}$s in row $\alpha_i$ of
 $A$ display the elements in the subsets $X_i$ and the $1^{\prime}$s in column $\beta_j$ display the ocurrences of the element $x_j$ among the subsets.
 \subsubsection{Configuration of incidence} 
 Let  $S=\{s_1\ldots, s_n \}$ an $n$-set and $S_1, \ldots, S_m$ be $m$ subsets of the $n$-set $S$ and ${\EuScript L}$ the
 $m\times n$ incidence matrix, for the configuration of subsets $S_1, \ldots, S_m$.
The pair
 \begin{equation}\label{confinc}
  \big(S,  S_i )_{i=1}^m
  \end{equation}
  we call {\it configuration of incidence of $S$}.
 If $(S^{\prime},  S_i^{\prime} \big)_{i=1}^m$, donde $S^{\prime}=\{s^{\prime}_1\ldots, s^{\prime}_n \}$, is other configuration of incidence then they are isomorphic  if and only if there is a bijection $$\psi: S\longrightarrow S^{\prime}$$
 $$ \psi(s_i)=s_i^{\prime}$$ such that $\psi(S_i)=S_i^{\prime}$ 
 for all $i=1, \ldots, m$ and note ${\EuScript L}={\EuScript L}^{\prime}$ where ${\EuScript L}$ and  ${\EuScript L}^{\prime}$ are $(m\times n)$-incidence matrices.\\
Let $(S^{\prime}, S_i^{\prime} \big)_{i=1}^m$ be an incidence configuration, with $S$ an $n$-set and $\{ a \} $
  a set of cardinality $1$ then using the cartesian product we define the  \it {cartesian incidence configuration }
  as follows
 \begin{equation}\label{confcart}
 a \times (S,  S_i )_{i=1}^m:=\bigg( \{ a \}\times S, \;\; S_{(a, i)}\bigg)_{i=1}^m 
 \end{equation}
where $S_{(a, i)}:=\{ a \}\times S_i$
\begin{lem}\label{incconfisomr}
The cartesian incidence configuration $\{ x \} \times (S,  S_i )_{i=1}^m$ is isomorphic to $(S,  S_i )_{i=1}^m$
and they have the same incidence matrix.
\end{lem} 
\begin{proof}
Since $|\{ x \} \times (S,  S_i )_{i=1}^m|=| (S,  S_i )_{i=1}^m|$ then the projection mapping 
$$\psi: \{ x \} \times (S,  S_i )_{i=1}^m\longrightarrow (S,  S_i )_{i=1}^m$$ $$(x,s) \longmapsto s$$ 
is one-one and clearly $\psi_{|\{ x \}\times S_i}=S_i$ thus the configurations are isomorphic.  Moreover
$(x, s)\in \{ x \}\times S_i$ if and only if $s\in S_i$ and so both configurations have the same incidence matrix.
\end{proof}
 \subsection{ Finite Field}
Let ${\mathbb F}_q$ be a finite field with $q$ elements, and denote
by $\overline{\mathbb F}_q$ an algebraic closure of ${\mathbb F}_q$. For a vector space $E$  over ${\mathbb F}_q$ of finite dimension $k$, let $\overline{E}=E\otimes_{{\mathbb F}_q}\overline{\mathbb F}_q$ be the corresponding vector space over the algebraically closed field $\overline{\mathbb F}_q$. We will be considering algebraic varieties in the projective space ${\mathbb P}(\overline{E})={\mathbb P}^{k-1}(\overline{\mathbb
F}_q)$. Recall that a projective variety $X\subseteq{\mathbb P}^{k-1}(\overline{\mathbb
F}_q)$ is defined over the finite field ${\mathbb F}_q$ if its
vanishing ideal can be generated by polynomials with coefficients in
${\mathbb F}_q$. 
\section{Family of Linear Relations of Contraction (FLRC).} 
Let  $E=(E, \langle , \rangle)$  symplectic vector space of dimension $2n$ defined over an arbitrary field ${\mathbb F}$ with symplectic base ${\mathcal B}=\{e_1, \ldots, e_{2n} \}$ that satisfies \ref{sympbase1}. 
For $\overline{\alpha}=(\overline{\alpha}_1,\ldots,\overline{\alpha}_n) \in I(n,2n)$ be a fixed index 
and let $\overline{\alpha}_k \in
\text{supp}(\overline{\alpha})$ denoted by 
\begin{equation}\label{setS}
{\mathcal S}=[2n]-\{\overline{\alpha}_k,  2n-\overline{\alpha}_k+1\} 
\end{equation}
and define 
\begin{equation}\label{indSet}
I_{\overline{\alpha}_k}(n-1,2n-2):=C_{n-1}\{ {\mathcal S} \}\\
\end{equation}
see \ref{setcomb1} for this notation. Let
\begin{equation}\label{phiSet}
 \phi:= \{(\beta,\overline{\alpha}_k)\in I(n, 2n): \beta\in
I_{\overline{\alpha}_k}(n-1,2n-2)\}
\end{equation}. 
for all $w=\sum_{\alpha\in I(n,2n)}P_\alpha e_\alpha \in \wedge^n E$ we do
\begin{gather}\label{expw}
 w= \sum_{\alpha\in \phi} P_\alpha e_\alpha
+ \sum_{\alpha\in \phi^c}P_\alpha e_\alpha = \sum_{\beta\in
I_{\overline{\alpha}_k}(n-1,2n-2)}P_{(\beta,\overline{\alpha}_k)}
e_{(\beta,\overline{\alpha}_k)} + \sum_{\alpha\in \phi^c}P_\alpha
e_\alpha.
\end{gather}
Let $e_{\overline{\alpha}_k}\in {\mathcal B}$ be a basic vector and let   $\ell=\langle e_{\overline{\alpha}_k}\rangle$ the isotropic line of $E$ generated by the  vector $e_{\overline{\alpha}_k}$ and 
\begin{equation}\label{Ulset}
{\mathcal{U}}(\ell)=\{W\in L(n, E): \ell\subseteq W\}
\end{equation} 
in  \cite[lemma 1.4.38]{bib3} it shows that ${\mathcal{U}}(\ell)$
 is symplectomorph to $L(n-1, \ell^{\perp}/\ell)$, where $\ell^{\perp}/\ell$ is a symplectic vector space of dimension $2n-2$, generated by the basis
$\{e_1+\ell, \ldots, \widehat{e}_{\overline{\alpha}_k}+\ell, \ldots, \widehat{e}_{2n-\overline{\alpha}_k+1}+\ell,\ldots,e_{2n}+\ell\}$, where  $\widehat{}$ means that the corresponding basic was omitted. The vector space $\wedge^{n-1}(\ell^{\perp}/\ell)$ is generated by $e_{\beta}+\ell:=(e_{\beta_1}+\ell)\wedge\cdots \wedge(e_{\beta_{n-1}}+\ell)$ where $\beta\in I_{\overline{\alpha}_k}(n-1, 2n-2)$.  Let  $w=\sum_{\beta\in I_{\overline{\alpha}_k}(n-1, 2n-2)}P_{\beta}(e_{\beta}+\ell)\in \wedge^{n-1}(\ell^{\perp}/\ell)$ in pl\"ucker coordinates, then $w\wedge e_{\overline{\alpha}_k}=\sum_{\beta\in I_{\overline{\alpha}_k}(n-1, 2n-2)}P_{(\beta, \overline{\alpha}_k)}e_{(\beta, \overline{\alpha}_k)}$, where $P_{(\beta, \overline{\alpha}_k)}=P_{\beta}$. With this notation, consider the injective linear transformation defined on generators
\begin{equation}\label{xx}
\begin{split}
 {_{-}\wedge e_{\overline{\alpha}_k}\atop \quad}{:\atop\quad} {\wedge^{n-1}(\ell^{\perp}/\ell) \atop e_{\beta}+\ell}
{\longrightarrow\atop \longmapsto}{\wedge^n E\atop e_{(\beta, \overline{\alpha}_k)}}{.\atop \quad}
\end{split}
\end{equation}
with $\beta\in I_{\overline{\alpha}_k}(n-1,2n-2)$.
So we have to 
\begin{equation}
 {_{-}\wedge e_{\overline{\alpha}_k}} \big(\sum_{\beta\in
I_{\overline{\alpha}_k}(n-1,2n-2)}p_{\beta}(e_{\beta}+\ell) \big\} \big)=\sum_{\beta\in
I_{\overline{\alpha}_k}(n-1,2n-2)}p_{(\beta,\overline{\alpha}_k)}e_{(\beta, \overline{\alpha}_k)} 
\end{equation}
with $p_{(\beta,\overline{\alpha}_k)}=p_{\beta}$.
As consequence we have $ {\mathcal{U}}(\ell)=_{-}\wedge e_{\overline{\alpha}_k}(L(n-1, \ell^{\perp}/\ell ))$ and so
\begin{equation}\label{Ulexterior}
 {\mathcal{U}}(\ell)  = L(n-1, \ell^{\perp}/\ell )\wedge e_{\overline{\alpha}_k}                                              
 \end{equation}
 so in pl\"ucker coordinates we have
 \begin{equation}\label{setU}
{\mathcal{U}}(\ell)=\big\{ w\in L(n, E): w=\sum_{\beta\in
I_{\overline{\alpha}_k}(n-1,2n-2)}p_{(\beta,\overline{\alpha}_k)}e_{({\beta}, \overline{\alpha}_k)} \big\} 
\end{equation}
Denote by $X_{\beta}:=\big( e_{\beta}+\ell \big)^*$ the basis vector of the dual vector space  $\big( \wedge^{n-1}\ell^{\perp}/\ell\big)^*$ and $X_{(\beta, \alpha_k)}:=e^*_{(\beta, \overline{\alpha}_k)}$
the basis vector  of dual vector space $\big(\wedge^n E \big)^*$.  Now with this notation
we define in generators an injective linear transformation
\begin{equation}\label{xx}
\begin{split}
\xi: (\wedge^{n-1}(\ell^{\perp}/\ell))^*\longrightarrow (\wedge^n E)^*\\
X_{\beta}  \longmapsto X_{(\beta, \overline{\alpha}_k)}
\end{split}
\end{equation}
with $\beta\in I_{\overline{\alpha}_k}(n-1,2n-2)$.\\
Clearly $Img\; \xi=\big\{h^{\prime}\in (\wedge^n E)^* : h^{\prime}=\sum_{\beta\in I_{\overline{\alpha}_k}(n-1, 2n-2)} A_{(\beta, \overline{\alpha}_k)}X_{(\beta, \overline{\alpha}_k)}\big\}$
and we denote by 
\begin{equation}\label{recimaghprime}
h^{\prime}_0:=\sum_{\beta \in I_{\overline{\alpha}_k}(n-1, 2n-2)}A_{\beta}X_{\beta}\in (\wedge^{n-1}\ell^{\perp}/\ell)^*
\end{equation}
to the unique linear functional on $(\wedge^{n-1}\ell^{\perp}/\ell)^*$ such that $\xi(h^{\prime}_0)=h^{\prime}$, where\\ 
$h^{\prime}=\sum_{\beta\in I_{\overline{\alpha}_k}(n-1, 2n-2)} A_{(\beta, \overline{\alpha}_k)}X_{(\beta, \overline{\alpha}_k)}$ y 
$A_{(\beta, \overline{\alpha}_k)}=A_{\beta}$. With this notation we have the following lemma.
\begin{lem}\label{h0null}
If $h^{\prime}\in Img\; \xi$ and  $h^{\prime}({\mathcal U}(\ell))=0$ then $h^{\prime}_0(L(n-1, \ell^{\perp}/\ell))=0$
\end{lem}
\begin{proof}
    Let  $h^{\prime}_0=\sum_{\beta \in I_{\overline{\alpha}_k}(n-1, 2n-2)}A_{\beta}X_{\beta}$  as in  \ref{recimaghprime}   and  $w=\sum_{\beta \in I_{\overline{\alpha}_k}(n-1, 2n-2)}P_{\beta}(e_{\beta}+\ell)$ an arbitrary element in 
   $L(n-1, \ell^{\perp}/\ell)$  then $_{-}\wedge e_{\overline{\alpha}_k}(w)=\sum_{\beta\in I_{\overline{\alpha}_k}(n-1, 2n-2)} P_{(\beta, \overline{\alpha}_k)}e_{(\beta, \overline{\alpha}_k)}\in {\mathcal U}(\ell)$ so then we have 
$h^{\prime}(_{-}\wedge e_{\overline{\alpha}_k}(w))=\sum_{\beta\in I_{\overline{\alpha}_k}(n-1, 2n-2)} A_{(\beta, \overline{\alpha}_k)}P_{(\beta, \overline{\alpha}_k)}=0$ this implies that 
$$\sum_{\beta\in I_{\overline{\alpha}_k}(n-1, 2n-2)} A_{\beta}P_{\beta}=0$$ since $P_{(\beta, \overline{\alpha}_k)}=P_{\beta}$   and 
$A_{(\beta, \overline{\alpha}_k)}=A_{\beta}$ we have $h^{\prime}_0(w)=\sum_{\beta\in I_{\overline{\alpha}_k}(n-1, 2n-2)} A_{\beta}P_{\beta}=0$ and  $h^{\prime}_0(L(n-1, \ell^{\perp}/\ell))=0$
\end{proof}
 \begin{lem}\label{indiceinsupp}
 Let $E=(E, \langle , \rangle)$ a symplectic vector space of dimension $2n$.\\ 
 Let $h=\sum_{\alpha \in I(n, 2n)}A_{\alpha}X_{\alpha}\in (\wedge^n E)^*$ nonzero such that $h(L(n, E))=0$
and  $A_{\overline{\alpha}}\neq 0$ a coefficient of $h$ then it exists $\{ \overline{\alpha}_i, \overline{\alpha}_j\}\subseteq supp \{ \overline{\alpha}\}$  such that $ \overline{\alpha}_i +  \overline{\alpha}_j = 2n+1$ 
  \end{lem}
 \begin{proof}
 Suppose that for each $\{ \overline{\alpha}_i, \overline{\alpha}_j\}\subseteq supp \{ \overline{\alpha}\}$ 
 you have to $\overline{ \alpha}_i + \overline{\alpha}_j\neq  2n+1$ this means that  $\langle e_{\overline{\alpha}_i}, e_{\overline{\alpha}_j} \rangle=0$ then  $e_{\overline{\alpha}}=e_{\overline{\alpha}_1}\wedge \cdots \wedge e_{\overline{\alpha}_n}\in L(n, E)$ and so then $h(e_{\overline{\alpha}})=A_{\overline{\alpha}}=0$ which is a contradiction. 
 Then there are $\{ \overline{\alpha}_i, \overline{\alpha}_j\}\subseteq supp \{ \overline{\alpha}\}$ such that $ \overline{\alpha}_i +  \overline{\alpha}_j = 2n+1$
  \end{proof}
\begin{thm}\label{funtiform}
Let $E$ simplectic vector space of  dimension $2n$, $r=\lfloor\frac{n-2}{2}\rfloor$, then
 for every $h=\Sigma_{\alpha\in I(n, 2n)}A_{\alpha} X_{\alpha}\in (\wedge^nE)^*$  such that  $h(L(n, E))=0$  satisfies one of the following :
\begin{description}
\item[ 1)] $h$  is of the form
\begin{equation}\label{funtional 1}
h=\sum_{P_{\theta}\in C_{\lfloor\frac{n-2}{2}\rfloor}(\Sigma_n)} h_{P_{\theta}} + \sum_{P_{(\theta, j)}\in C_{\lfloor\frac{n-3}{2}\rfloor}(\Sigma_n)\times \{ j\}} h_{(P_{\theta}, j)}\in (\wedge^n E)^*
\end{equation}
where
\begin{equation}\label{funtional 12}
h_{P_{\theta}}=\sum_{i=1}^nA_{(P_{\theta_1}, \dots, P_{\theta_{\lfloor\frac{n-2}{2}\rfloor}} P_i)} X_{(P_{\theta_1}, \dots, P_{\theta_{\lfloor\frac{n-2}{2}\rfloor}} P_i)}
\end{equation}
for  $P_{\theta}=(P_{\theta_1}, \dots, P_{\theta_{\lfloor\frac{n-2}{2}\rfloor}} )\in C_{\lfloor\frac{n-2}{2}\rfloor}(\Sigma_n)$,
or
\begin{equation}\label{funtional 13}
h_{(P_{\theta}, j)}=\sum_{i=1}^nA_{(P_{\theta_1}, \dots, P_{\theta_{\lfloor\frac{n-3}{2}\rfloor}}, j,  P_i)} X_{(P_{\theta_1}, \dots, P_{\theta_{\lfloor\frac{n-3}{2}\rfloor}}, j,  P_i)}
\end{equation}
for  $(P_{\theta}, j)=(P_{\theta_1}, \dots, P_{\theta_{\lfloor\frac{n-3}{2}\rfloor}}, j )\in C_{\lfloor\frac{n-3}{2}\rfloor}(\Sigma_n)\times \{ j \}$, for $j=1, \ldots, n$
\item[ 2)] If $h$ is not of the form \ref{funtional 1}, is not of the form \ref{funtional 12} or not of the form \ref{funtional 13}, in this case $h$ has at least one non-zero coefficient of the form
\begin{equation}\label{coefficient}
A_{(\overline{\alpha}_1, \ldots, \overline{\alpha}_{2 l}, P_{\theta})}
\end{equation}
 with   
 $1\leq \overline{\alpha}_1< \ldots, <\overline{\alpha}_{2 l}\leq n$ such that $\overline{\alpha}_i+ \overline{\alpha}_j\neq 2n+1$ for all $1\leq i < j \leq n$, $1\leq \ell \leq r-2$  and $P_{\theta}=(P_{\theta_1}, \dots, P_{\theta_{\lfloor\frac{n - 2\ell}{2}\rfloor}})\in C_{\lfloor\frac{n-2 \ell}{2}\rfloor}(\Sigma_{\overline{\alpha}_1, \ldots, \overline{\alpha}_{2 \ell}})$.
\end{description} 
 \end{thm}
 \begin{proof}
 By lemma \ref{indiceinsupp} each coefficient $A_{\alpha}\neq 0$ de $h$ satisfies $supp\{\alpha\}\cap \Sigma_n\neq\emptyset$, then   $supp\{\alpha \}\subset \Sigma_n$ for all  $A_{\alpha}\neq 0$ and so we have
  $h$ is of the form \ref{funtional 1}.  Now if $A_{\alpha}\neq 0$ satisfies that $1\leq |supp\{\alpha\}\cap \Sigma_n|<\frac{n-2\ell}{2}$ then $h$ is of the form \ref{coefficient}. 
 \end{proof} 
\begin{lem}\label{lem2}
Let $h \in (\wedge^n E)^*$ such that $h(L(n,E)) = 0$ and $A_{(\overline{\alpha}_1, \ldots, \overline{\alpha}_{2 l}, P_{\theta})}\neq 0$ a coefficient that satisfies \ref{coefficient} then there are
$h^{\prime}, h^{\prime\prime}\in (\wedge^n E)^*$ such that $h=h^{\prime} + h^{\prime\prime}$, $h^{\prime}\neq 0$ and $h^{\prime}({\mathcal{U}}(\ell))=h''({\mathcal{U}}(\ell))=0$  
\end{lem}
\begin{proof} 
If $A_{(\overline{\alpha}_1, \ldots, \overline{\alpha}_{2k}, P_{\theta})}$ a coeficient 
 nonzero with $\overline{\alpha}=(\overline{\alpha}_1, \ldots, \overline{\alpha}_{2 l}, P_{\theta})\in I(n-2, 2n)$,  
 such that $\overline{\alpha}_i+ \overline{\alpha}_j\neq 2n+1$   y $P_{\theta}=(P_{\theta_1}, \dots, P_{\theta_{\lfloor\frac{n-2 l}{2}\rfloor}})\in C_{\lfloor\frac{n-2 l}{2}\rfloor}(\Sigma_{(\overline{\alpha}_1, \ldots, \overline{\alpha}_{2 l}})$.
Let $\overline{\alpha}_k\in \{\overline{\alpha}_1, \ldots, \overline{\alpha}_{2 l}\}$, an arbitrary element and $\phi=\{(\beta,\overline{\alpha}_k)\in I(n,2n): \beta \in I_{\overline{\alpha}_k}(n-1,2n-2)\}$ as in   \ref{phiSet}
then clearly  $(\overline{\alpha}_1, \ldots, \overline{\alpha}_k, \ldots, \overline{\alpha}_{2\ell}, P_{\theta})\in \phi$
as $P_{\theta}\in \Sigma_{(\overline{\alpha}_1, \ldots, \overline{\alpha}_{2 l}})$.
If $h=\Sigma_{\alpha\in I(n, 2n)}A_{\alpha}X_{\alpha}\in (\wedge^n E)^*$ then clearly $h=\Sigma_{\alpha\in \phi}A_{\alpha}X_{\alpha}+\Sigma_{\alpha\in \phi^c}A_{\alpha}X_{\alpha}$ we define the linear functionals
$h^{\prime}:=\Sigma_{\alpha\in \phi}A_{\alpha}X_{\alpha}$  and $h^{\prime\prime}:=\Sigma_{\alpha\in \phi^c}A_{\alpha}X_{\alpha}$, both in $ (\wedge^n E)^*$, such that $h=h^{\prime}+h^{\prime\prime}$ so
 $h^{\prime} = \sum_{\beta \in I_{\overline{\alpha}_k}(n-1,2n-2)} A_{(\beta,\overline{\alpha}_k)}X_{(\beta,\overline{\alpha}_k)}$ and $h^{\prime\prime}=\sum_{\alpha\in \phi^c} A_{\alpha}X_{\alpha}$.
Note that $h^{\prime}\neq 0$, because as we mentioned before $(\overline{\alpha}_1, \ldots,\overline{\alpha}_k, \ldots  \overline{\alpha}_{2\ell}, P_{\theta})\in \phi$, from \ref{setU} we have that
$h^{\prime\prime}({\mathcal{U}}(\ell))=0$ that is $h({\mathcal{U}}(\ell))=h^{\prime}({\mathcal{U}}(\ell))$, since $h(L(n,E))=0$ we obtain
$h^{\prime}({\mathcal{U}}(\ell))=0$.
\end{proof}
As in \ref{notbasis} we denote by $\beta_{rs}=(\beta_1, \ldots,\widehat{\beta}_r, \ldots, \widehat{\beta}_s, \ldots, \beta_{n-1} )\in I_{\overline{\alpha}_k}(n-3, 2n-2)$, where \;\; $\widehat{}$\;\; means that the corresponding term is omitted. We define
\begin{equation}\label{subpi}
\Pi_{(\beta_{rs},\overline{\alpha}_k)}:= \sum_{i=1}^n
X_{i(\beta_{rs},\overline{\alpha}_k)(2n-i+1)} 
\end{equation}
functional in $(\bigwedge^n E)^*$
\begin{lem}\label{ULinKer}
 For all $\beta_{rs}\in I_{\overline{\alpha}_k}(n-3, 2(n-1))$ 
 where $1\leq r < s \leq 2n $ and $r,s\in {\mathcal S}$ with  $S$ as in \ref{setS}      
let $\Pi_{(\beta_{rs},\overline{\alpha}_k)}$ as \ref{subpi} then \\ 
 $$ {\mathcal{U}}(\ell) \subseteq \bigcap_{\beta_{rs}\in I_{\overline{\alpha}_k}
(n-3,2n-2)} \ker \Pi_{(\beta_{rs},\overline{\alpha}_k)}.$$
\end{lem}
\begin{proof}
Let $w\in {\mathcal{U}}(\ell)$ be from  \ref{setU} we
have $ w=\sum_{\beta \in I_{\overline{\alpha}_k}(n-1,2n-2)}
P_{(\beta,\overline{\alpha}_k)} e_{(\beta,\overline{\alpha}_k)} $ 
and $f(w)=0 $ for $f$ contraction map, we have $$f(w) = \sum_{\beta
\in I_{\overline{\alpha}_k}(n-1,2n-2)}
P_{(\beta,\overline{\alpha}_k)} f(e_{(\beta,\overline{\alpha}_k)})=
$$ $$ = \sum_{\beta \in I_{\overline{\alpha}_k}(n-1,2n-2)}
P_{(\beta,\overline{\alpha}_k)} \big(\sum_{1\leq r < s \leq n}  \langle
e_{\alpha_r}, e_{\alpha_s}\rangle (-1)^{r+s-1}
e_{(\beta,\overline{\alpha}_k)_{rs}}\big)  $$
 $$= \sum_{1\leq r < s \leq
n}(\sum_{1\leq \rho_1 < \rho_2 \leq n}
P_{(\beta_{rs},\overline{\alpha}_k,\rho_1, \rho_2)} \langle
e_{\rho_1}, e_{\rho_2}\rangle (-1)^{\rho_1+\rho_2-1})e_{(\beta, \overline{\alpha}_k)_{rs}}
= 0 $$
 then $$  \sum_{1\leq \rho_1 < \rho_2 \leq n}
P_{(\beta_{rs},\overline{\alpha}_k,\rho_1, \rho_2)} (-1)^{\rho_1+\rho_2-1} \langle
e_{\rho_1}, e_{\rho_2}\rangle =0.$$
where $\beta_{rs}\in I_{\overline{\alpha}_k}(n-3,2n-2)$ and as stated before
$(\beta_{rs},\overline{\alpha}_k,\rho_1, \rho_2)\in I(n, 2n)$.
 Now $\langle e_{\rho_1},
e_{\rho_2}\rangle = 1$ iff $\rho_1+\rho_2 = 2n+1$, note that with this condition we have 
$(-1)^{\rho_1+\rho_2-1}=1$. Renaming $\rho_1 = i$,
we have $\rho_2 = 2n-i+1$.
 Then
\begin{align*} 
  \sum_{1\leq \rho_1 < \rho_2\leq n}
P_{(\beta_{rs},\overline{\alpha}_k)\rho_1\rho_2}= \sum_{i=1}^n
P_{i(\beta_{rs}\overline{\alpha}_k) 2n-i+1} = 0
\end{align*}
for all
$\beta_{rs} \in I_{\overline{\alpha}_k}(n-3,2n-2)$, that is
${\mathcal{U}}(\ell) \subseteq \ker
\Pi_{(\beta_{rs},\overline{\alpha}_k)}$ 
\end{proof}
\begin{rem}\label{MorPi}
Observe que  
$$\{\Pi_{(\beta_{rs}, \overline{\alpha}_k)}: \beta_{rs} \in I_k(n-3,
2n-2)\} \subseteq \{\Pi_{\alpha_{rs}}: \alpha_{rs} \in I(n-2,
2n)\}$$ para todo $1\leq r< s \leq 2n-2$.
\end{rem}
\begin{lem}\label{lem4}If $E$ is a symplectic vector space of dimension 4, defined on an arbitrary field. Then the
unique set of linear functionals  $h\in (\wedge^2 E)^* $ that annuls
 $L(2, E)$ are of the form $h=A\Pi $
with $A$ on scalar and $\Pi=X_{14} + X_{23}$. 
\end{lem}
\begin{proof}
Suppose that $h=A_{12}X_{12}+A_{13}X_{13}+A_{14}X_{14}+A_{23}X_{23}+A_{24}X_{24}+A_{34}X_{34}\in (\wedge^2E)^*$
and supongamos que $h(L(2, E))=0$ clearly we have to $\{e_{12}, e_{13}, e_{24}, e_{34} \}\subset L(2, E)$ so
$h(e_{12})=h(e_{13})=h(e_{24})=h(e_{34})=0$ and this $A_{12}=A_{13}=A_{24}=A_{34}=0$, it is also easy to see that $e_{12}+e_{14}-e_{23}+e_{34}\in L(2, E)$ because its coordinates satisfy the  equations \ref{lagrzeroset2}   more over 
\begin{align*}
h(e_{12}+e_{14}-e_{23}+e_{34})&=A_{12}\cdot 1+A_{14}\cdot 1-A_{23}\cdot 1+A_{34}\cdot 1\\
                                                   &=A_{14}-A_{23}\\
                                                   &=0
\end{align*}
 consequently $A_{14}=A_{23}$ and this implies that
$h=A(X_{14} + X_{23})$
where $A=A_{14} = A_{23}$ is a nonzero constant.
\end{proof}
\begin{lem}\label{indodd}
Let $\overline{E}$ a symplectic vector space of dimension $6$, let $j=1, \ldots, 6$ and $h_j=\sum_{i=1}^3A_{(P_i, j)}X_{(P_i, j)}\in(\wedge^3\overline{E})^*$ where $A_{(P_i, j)}\neq 0$ if $|supp \{ P_i, j\}|=3$\\
If $h_j(L(3, \overline{E}))=0$ then $A_{(P_i, j)}=A$ a nonzero constant for all $i=1, \ldots, 3$ and $j=1, \ldots, 6$
\end{lem}
\begin{proof}
Let $\{e_1, \ldots, e_6 \}$ a symplectic basis that satisfies \ref{sympbase1},
without loss of generality it is enough to do it for $j=1$, then  
\begin{align*}
h_{(P_{\theta},1)}&=A_{( P_2, 1)}X_{ (P_2, 1)}+A_{ (P_3, 1)}X_{ (P_3, 1)}\\
                          &=(A_{P_2}X_{ P_2 }+A_{ P_3 }X_{ P_3 })\wedge X_1 \\
                          &=h_{P_{\theta}}\wedge X_1
\end{align*}
where $h_{P_{\theta}}=A_{P_2}X_{ P_2 }+A_{ P_3 }X_{ P_3 }$.
 If  $E=\langle e_2, e_5\rangle \bot \langle e_3, e_4\rangle$ a symplectic vector subspace of
$\overline{E}$ of dimension $4$.
Let $\ell=\langle e_1\rangle\subset \overline{E}$ be the simplectic line then  ${\mathcal{U}}(\ell)= L(2, \ell^{\bot}/\ell)\wedge e_1$ \\
Note that $L(2, \ell^{\bot}/\ell)$ is given by the homogeneous equations
\begin{align*}\label{lagrzeroset022} 
\begin{cases}
X_{23}X_{45}-X_{24}X_{35}-X_{34}X_{25}&=0\\
X_{25}+X_{34}&=0
\end{cases}
\end{align*}
so $w=e_{23}+e_{25}-e_{34}+e_{45}\in L(2, \ell^{\bot}/\ell)$, and $h(w)=A_{(P_2, 1)}-A_{(P_3, 1)}=0$
then $A_{(P_2, 1)}=A_{(P_3, 1)}=A$ a nonzero constant.
\end{proof}
\begin{prop}\label{morphhom}
Let $E$ be a simplectic vector space of dimension $2(n+1)$ and
 $P_{\theta}\in C_{\lfloor\frac{n-1}{2}\rfloor}(\Sigma_{n+1})$. Let  $h_{P_{\theta}}=\Sigma_{i=1}^{n+1}A_{P_{\theta}P_i}X_{P_{\theta}P_i}\in (\wedge^{n+1} E)^*$ is as in \ref{funtional 12} and
 $h_{P_{\theta}}(L(n+1, E))=0$ then $A_{P_{\theta}P_i}=A$ 
 a nonzero constant.
\end{prop}
\begin{proof} 
The proof is by induction on $n$.\\
$I)$ The cases  $n=2$ y $n=3$  were proved in lemma \ref{lem4} and lemma \ref{indodd} respectively .\\
$II)$ Induction hypothesis: Let $E$ simplectic vector space of dimension $n$.\\
Let $P_{\theta}\in C_{\lfloor\frac{n-2}{2}\rfloor}(\Sigma_n)$ and let $h_{P_{\theta}}=\Sigma_{i=1}^nA_{P_{\theta}P_i}X_{P_{\theta}P_i}\in (\wedge^n E)^*$  if $h_{P_{\theta}}(L(n, E))=0$ then $A_{P_{\theta}P_i}=A$ a nonzero constant for each  $i\in \{1, \ldots, n \}$ such that  $|supp\{(P_{\theta}, P_i) \}|=\lfloor\frac{n}{2} \rfloor$.\\
$III)$ $a)$ Suppose that the induction hypothesis holds for $n$ even integer.\\
 Let $E$ simplectic vector space of dimension $2(n+1)$,  $1\leq j \leq n+1$ and\\ 
 $(P_{\theta}, j)=(P_{\theta_1}, \ldots, P_{\theta_{\frac{n-2}{2}}}, j)\in C_{\frac{n-2}{2}}(\Sigma_{n+1}-\{P_j\})\times \{ j \}$.\\
 Let $h_{(P_{\theta}, j)}=\sum_{i=1}^{n+1}A_{(P_{\theta_1}, \ldots, P_{\theta_{\frac{n-2}{2}}}, j, P_i)}X_{(P_{\theta_1}, \ldots, P_{\theta_{\frac{n-2}{2}}}, j, P_i)}\in (\wedge^{n+1} E)^*$ such that \\
 $h_{(P_{\theta}, j)}(L(n+1, E))=0$ 
\begin{align*}
h_{(P_{\theta}, j)}&=\sum_{i=1, i\neq j}^{n+1}A_{(P_{\theta_1}, \ldots, P_{\theta_{\frac{n-2}{2}}}, j, P_i)}X_{(P_{\theta_1}, \ldots, P_{\theta_{\frac{n-2}{2}}}, j, P_i)}\\
                         &=\big(\sum_{i=1, i\neq j}^{n+1}A_{(P_{\theta_1}, \ldots, P_{\theta_{\frac{n-2}{2}}}, j, P_i)}X_{(P_{\theta_1}, \ldots, P_{\theta_{\frac{n-2}{2}}}, P_i)}\big)\wedge X_j\\
                         &= h_{P_{\theta}}\wedge X_j
\end{align*}
then
\begin{equation}\label{hptheta}
h_{(P_{\theta}, j)}=h_{P_{\theta}}\wedge X_j
\end{equation}
where 
\begin{equation}\label{hptheta11}
h_{P_{\theta}}:= \sum_{i=1}^{n+1}A_{(P_{\theta_1}, \ldots, P_{\theta_{\frac{n-2}{2}}}, j, P_i)}X_{(P_{\theta_1}, \ldots, P_{\theta_{\frac{n-2}{2}}}, P_i)} 
\end{equation}
and \\
$P_{\theta}=(P_{\theta_1}, \dots, P_{\theta_{\frac{n-2}{2}}})\in C_{\frac{n-2}{2}}(\Sigma_{n+1}-\{P_j \})$.
Let $\ell=\langle e_j \rangle\subset E$  isotropic line, and as in \ref{Ulexterior}  we have 
\begin{equation}\label{hp11theta11}
{\mathcal{U}}(\ell)=L(n, \ell^{\bot}/\ell) \wedge e_j \subseteq L(n+1, E)
\end{equation}
clearly $L(n, \ell^{\bot}/\ell))\wedge e_j=L(n+1, E)\cap (\wedge^n \ell^{\bot}/\ell)\wedge e_j$ for the symplectic vector space $\ell^{\bot}/\ell$, of dimension $2n$ y $h_{P_{\theta}}\in (\wedge^n \ell^{\bot}/\ell)^*$. Moreover using \ref{hptheta} and \ref{hp11theta11} we have
\begin{align*}
h_{P_{\theta}}(L(n, \ell^{\bot}/\ell))&=h_{P_{\theta}}(L(n, \ell^{\bot}/\ell))\cdot X_j(e_j)\\
                                                     &=(h_{P_{\theta}}\wedge X_j)(L(n, \ell^{\bot}/\ell))\wedge e_j)\\
                                                     &=h_{(P_{\theta},j)}(L(n+1, E)\cap (\wedge^n \ell^{\bot}/\ell)\wedge e_j)\\
                                                     &=0                                                          
\end{align*}
then $h_{P_{\theta}}(L(n, \ell^{\bot}/\ell))=0$ and so by induction hypothesis, $A_{(P_{\theta_1}, \ldots, P_{\theta_{\frac{n-2}{2}}}, j, P_i)}=A$
a nonzero constant.\\
$b)$ Suppose that the induction hypothesis holds for $n$ odd.
Let $E$ simplectic vector space of dimension $2(n+1)$, 
let $P_{\theta}=(P_{\theta_1}, \ldots P_{\theta_{\frac{n-1}{2}}})\in C_{\frac{n-1}{2}}(\Sigma_{n+1})$.\\
Let  $h_{P_{\theta}}=\sum_{i=1}^{n+1}A_{(P_{\theta_1}, \ldots P_{\theta_{\frac{n-1}{2}}}, P_i)}X_{(P_{\theta_1}, \ldots P_{\theta_{\frac{n-1}{2}}}, P_i)}\in (\wedge^{n+1}E)^*$ \\
Then
\begin{align*}
h_{P_{\theta}}&=\sum_{i=1}^{n+1}A_{(P_{\theta_1}, \ldots P_{\theta_{\frac{n-1}{2}}}, P_i)}X_{(P_{\theta_1}, \ldots P_{\theta_{\frac{n-1}{2}}}, P_i)}\\
                      &=(\sum_{i=1}^{n+1}A_{(P_{\theta_1}, \ldots P_{\theta_{\frac{n-1}{2}}}, P_i)}X_{(P_{\theta_2}, \ldots P_{\theta_{\frac{n-1}{2}}}, \theta_1,  P_i)})\wedge X_{2n-\theta_1+1}
\end{align*}
Let $(P_{\epsilon}, \theta_1):=(P_{\theta_2}, \ldots, P_{\theta_{\frac{n-1}{2}}}, \theta_1)\in C_{\frac{n-3}{2}}(\Sigma_{n+1}-\{P_1 \})\times \{\theta_1 \}$, then
\begin{equation}\label{seconpthet}
h_{P_{\theta}}=h_{(P_{\epsilon}, \theta_1)}\wedge X_{2n-\theta_1+1}
\end{equation}
where 
\begin{equation}\label{s11econpthe}
h_{(P_{\epsilon}, \theta_1)}:=\sum_{i=1, i\neq \theta_1}^{n+1}A_{(P_{\theta_1}, \ldots P_{\theta_{\frac{n-1}{2}}}, P_i)}X_{(P_{\theta_2}, \ldots P_{\theta_{\frac{n-1}{2}}}, \theta_1, P_i)}\in(\wedge^n\ell^{\bot}/\ell)^*
\end{equation}
if $\ell=\langle e_{2n-\theta_1+1}\rangle$ then $h_{(P_{\epsilon}, \theta_1)}\in(\wedge^n\ell^{\bot}/\ell)^*$ and \\
\begin{equation}\label{seconpthe}
{\mathcal{U}}(\ell)=L(n, \ell^{\bot}/\ell) \wedge e_{\theta_1}=L(n+1, E)\cap (\wedge^n \ell^{\bot}/\ell)\wedge e_{\theta_1}
\end{equation}
where $\ell^{\bot}/\ell$ is the symplectic vector space of dimension $2n$.\\
Also using \ref{seconpthet} and \ref{seconpthe}
\begin{align*}
h_{(P_{\epsilon}, \theta_1)}(L(n, \ell^{\bot}/\ell))&=h_{(P_{\epsilon}, \theta_1)}(L(n, \ell^{\bot}/\ell))\cdot X_{2n-\theta_1+1}( e_{2n-\theta_1+1})\\
                                           &=(h_{(P_{\epsilon}, \theta_1)}\wedge X_{2n-\theta_1+1})(L(n, \ell^{\bot}/\ell)\wedge e_{2n-\theta_1+1})\\
                                           &=h_{P_{\theta}}(L(n+1, E)\cap (\wedge^n \ell^{\bot}/\ell)\wedge e_{2n-\theta_1+1})\\
                                           &=0
\end{align*}
then $h_{(P_{\epsilon}, \theta_1)}$, satisfies the conditions of the induction hypothesis  and so $A_{(P_{\theta_1}, \ldots P_{\theta_{\frac{n-1}{2}}}, P_i)}=A$ a nonzero constant.
\end{proof} 
\begin{cor}\label{sumhs1}
Let $E$ be a simplectic vector space of dimension $2n$.\\
$a)$ Let $h_{P_{\theta}}=\sum_{i=1}^{n}A_{P_{\theta, P_i}}X_{P_{\theta, P_i}}\in (\wedge^n E)^*$ as in \ref{funtional 12} such that
$h_{P_{\theta}}(L(n, E))=0$  then $h_{P_{\theta}}\in \langle \Pi_{\alpha_{rs}}: \alpha_{rs}\in I(n-2, 2n) \rangle$.\\
$b)$ With the notation of \ref{funtional 1} let $h=\sum_{P_{\theta}\in C_{\lfloor\frac{n-2}{2}\rfloor}(\Sigma_n)} h_{P_{\theta}}\in (\wedge^n E)^*$, such that $h(L(n, E))=0$ then  $h\in \langle \Pi_{\alpha_{rs}}:  \alpha_{rs}\in I(n-2, 2n) \rangle_{{\mathbb F}}$
\end{cor}
\begin{proof}
$a)$ If $h_{P_{\theta}}(L(n, E))=0$ then by lemma \ref{morphhom} we have $h_{P_{\theta}}=A(\sum_{i=1}^{n} X_{P_{\theta}P_i})=A\Pi_{P_{\theta}}$ since all $A_{P_{\theta}P_i}=A$, so $h_{P_{\theta}}\in \langle \Pi_{\alpha_{rs}}: \alpha_{rs}\in I(n-2, 2n) \rangle$\\
$b)$ Let $(P_{\theta}, j)=(P_{\theta_1}, \ldots P_{\theta_{\frac{n-2}{2}}}, j)\in C_{\lfloor\frac{n-2}{2}\rfloor}(\Sigma_n)\times \{j\}$
let $\ell=\langle e_j \rangle\subset E$ isotropic line, so $h_{\big|{\mathcal{U}}(\ell)}=h_{(P_{\theta}, j)}$ and
clearly $h_{(P_{\theta}, j)}({\mathcal{U}}(\ell))=h({\mathcal{U}}(\ell))=0$ then  we have to exist $h_{P_{\theta}}\in (\wedge^{n-1}(\ell^{\bot}/\ell))^*$ as in \ref{hptheta11} such that $h_{(P_{\theta}, j)}=h_{P_{\theta}}\wedge X_j$
and $h_{P_{\theta}}(L(n-1, \ell^{\bot}/\ell))=0$ then by $a)$ from this corollary we have $h_{P_{\theta}}=\sum_{\epsilon_{rs}\in I(n-3, 2(n-1))}c_{\epsilon_{rs}}\Pi_{\epsilon_{rs}}$ so we have
\begin{align*}
h_{(P_{\theta}, j)}&=h_{P_{\theta}}\wedge X_j\\
                           &=\big(\sum_{\epsilon_{rs}\in I(n-3, 2n-2)}c_{\epsilon_{rs}}\Pi_{\epsilon_{rs}}\big)\wedge X_j\\
                           &=\sum_{(\epsilon_{rs}, j)\in I(n-3, 2n-2)\times \{j\}}c_{(\epsilon_{rs},j)}\Pi_{(\epsilon_{rs},j)}
\end{align*}
where $c_{(\epsilon_{rs}, j)}=c_{\epsilon_{rs}}$ so  then, $h_{(P_{\theta}, j)}\in \langle \Pi_{\alpha_{rs}}:  \alpha_{rs}\in I(n-2, 2n) \rangle$ \\
Now let $P_{\theta}=(P_{\theta_1}, \ldots, P_{\theta{\lfloor\frac{n-2}{2}\rfloor}})\in C_{\lfloor\frac{n-2}{2}\rfloor}(\Sigma_n)$ and let $\ell=\langle e_{\theta_{1} }\rangle\subset E$ isotropic line, then  we have $h_{|{\mathcal{U}}(\ell)}=h_{P_{\theta}}$ so $h_{P_{\theta}}({\mathcal{U}}(\ell))=h({\mathcal{U}}(\ell))=0$
and let  $(P_{\epsilon}, \theta_1):=(P_{\theta_2}, \ldots, P_{\theta_{\lfloor\frac{n-2}{2}}\rfloor}, \theta_1)\in C_{\lfloor \frac{n-3}{2}\rfloor}(\Sigma_n-\{P_{\theta_1} \})\times \{\theta_1 \}$ then let $h_{(P_{\epsilon}, \theta_{1})}\in (\wedge^{n-1}\ell^{\bot}/\ell)^*$ as in \ref{s11econpthe} such that  $h_{P_{\theta}}=h_{(P_{\epsilon}, \theta_1)}\wedge X_{2n-\theta_1+1}$ and so we have $h_{(P_{\epsilon}, \theta_1)}(L(n-1, \ell^{\bot}/\ell))=0$ where  $\ell^{\bot}/\ell$ is the symplectic vector space of dimension $2(n-1)$, then by $a)$ from this corollary we have \\
$h_{(P_{\epsilon}, \theta_1)}=\sum_{(\epsilon_{rs}, \theta_1)\in I(n-3, 2(n-1))\times \{\theta_1 \}}c_{(\epsilon_{rs}, \theta_1)}\Pi_{(\epsilon_{rs}, \theta_1)}$ so we have to
\begin{align*}
h_{P_{\theta}}&=h_{(P_{\epsilon}, \theta_1)}\wedge X_{2n-\theta_1+1}\\
                           &=(\sum_{(\epsilon_{rs}, \theta_1)\in I(n-4, 2n-2)\times \{\theta_1 \}}c_{(\epsilon_{rs}, \theta_1)}\Pi_{(\epsilon_{rs}, \theta_1)})\wedge X_{2n-\theta_1+1}\\
                           &=\sum_{(\epsilon_{rs}, P_1)\in I(n-2, 2n)}c_{(\epsilon_{rs}, P_1)}\Pi_{(\epsilon_{rs}, P_1)}
\end{align*}
where $c_{(\epsilon_{rs}, P_1)}=c_{(\epsilon_{rs}, \theta_1)}$ so then we have to $h_{P_{\theta}}\in \langle \Pi_{\alpha_{rs}}:  \alpha_{rs}\in I(n-2, 2n) \rangle_{{\mathbb F}}$ where  $\alpha_{rs}=(\epsilon_{rs}, \theta_1)$
and therefore $h\in \langle \Pi_{\alpha_{rs}}:  \alpha_{rs}\in I(n-2, 2n) \rangle_{{\mathbb F}}$
\end{proof}
\begin{lem}\label{UniP2}
Let $E$ be a symplectic space of dimension $2n$ defined on an arbitrary field ${\mathbb F}$. 
If $h \in (\wedge^n E)^*$ such that $h(L(n,E)) = 0$  
then $h\in \langle \Pi_{\alpha_{rs}}:  \alpha_{rs}\in I(n-2, 2n) \rangle_{{\mathbb F}}$ 
\end{lem}
\begin{proof}
If $h \in (\wedge^n E)^*$ is as in \ref{funtional 1} then the theorem follows from the corollary \ref{sumhs1}.\\ 
If $h \in (\wedge^n E)^*$ it is not like \ref{funtional 1} so $h$ has at least a nonzero coefficient of the form $A_{(\overline{\alpha}_1, \ldots, \overline{\alpha}_{2 l}, P_{\theta})}$, see \ref{coefficient},  in this case  the proof is given by induction in $n$.\\
If $n=2$ then  by lemma \ref{lem4} we have that the only functionals that cancel $L(2, E)$ are the $h=A(X_{14} + X_{23})$ and this gives our induction basis.
The induction step: for $F$ a simplectic vector space of dimension $2n-2$,  
and every polynomial $h\in (\wedge^{n-1}F)^*$ different from zero, which cancels  $L(n-1, F)$ then the polynomial    $h$ is a linear combination of polynomials in $\langle \Pi_{\alpha_{rs}}: \alpha_{rs}\in I(n-3, 2n-2)\rangle_{{\mathbb F}}$.\\
Let $h=\sum_{\alpha\in I(n, 2n)}A_{\alpha}X_{\alpha} \in (\wedge^n E)^*$  such that $h(L(n, E))=0$ and $A_{(\overline{\alpha}_1, \ldots, \overline{\alpha}_{2 l}, P_{\theta})}\neq 0$ a coefficient that satisfies \ref{coefficient}
let $\overline{\alpha}_k\in \{ \overline{\alpha}_1, \ldots,  \overline{\alpha}_{2k}\}$,  $\phi$ be as in \ref{phiSet} and $\ell =\langle e_{\overline{\alpha}}
\rangle\subseteq E$ isotropic line of $E$,  then by lemma \ref{lem2} we have that there are homogeneous polynomials  $h^{\prime}, h^{\prime\prime}\in (\wedge^n E)^*$  such that $h^{\prime}\neq 0$ and $h^{\prime}({\mathcal{U}}(\ell))=h^{\prime\prime}({\mathcal{U}}(\ell))=0$, with notation of \ref{recimaghprime},
by induction step and lemma \ref{h0null}  you have to
 $h_0^{\prime}=\sum_{\beta_{rs}\in I_{\overline{\alpha}_r}(n-3, 2n-2)}A_{\beta_{rs}}\Pi_{\beta_{rs}}$  this implies that $h^{\prime}=\xi(h_0^{\prime})=\sum_{\beta_{rs}\in I(n-3, 2n-2)}A_{(\beta_{rs}, \overline{\alpha}_r )}\Pi_{(\beta_{rs}, \overline{\alpha}_r )}$ and by Remark \ref{MorPi} you have to $h^{\prime}\in \langle \Pi_{\alpha_{rs}}: \alpha_{rs}\in I(n-2, 2n)  \rangle$. Now if $h^{\prime\prime}\neq 0$ and  $h^{\prime\prime}\notin \langle \Pi_{\alpha_{rs}} :\alpha_{rs}\in I(n-2, 2n) \rangle$ since $h^{\prime \prime}(L(n, E))=0$ then one proceeds in the same as before, and the process ends in a finite number of steps.
   \end{proof}
The proof of the following theorem follows directly from the lemma \ref{UniP2}
 \begin{thm}\label{UniP22}
 The Family of Linear Relations of Contraction (FLRC) of $L(n, E)$,  definition  \ref{relinpluker}, is the set $\{ \Pi_{\alpha_{rs}} : \alpha_{rs}\in I(n, 2n) \}$ where $ \Pi_{\alpha_{rs}}$ is as in  \ref{FuncLinPI}, and is the only family with this property.
  \end{thm}
\qed
\section{The pl\"ucker matrix of the Lagrangian-Grassmannian }\label{recMatriz2}
\begin{defn}
Let  $E$  symplectic vector space of dimension $2n$ defined over an arbitrary field ${\mathbb F}$ arbitrary.
The  matrix
\begin{equation}\label{Bmatrixb0}
B_{L(n, E)}
\end{equation}
of order $C_{n-2}^{2n}\times C_{n}^{2n}$ associated to the linear equations system
 $\{\Pi_{\alpha_{rs}}: \alpha_{rs}\in I(n-2,2n)\}$ we will call you {\it the pl\"ucker matrix of the Lagrangian-Grassmannian}
 \end{defn}
 The main result of this section is the following theorem \ref{resinic000}:\\
 \begin{thm}\label{resinic000}
 Let  $E$  symplectic vector space of dimension $2n$ and let $r_n=\lfloor\frac{n+2}{2}\rfloor$,  then there exists a family $${\EuScript A}=\{ {\EuScript L}_{r_n}, {\EuScript L}_{r_n-1}, \ldots, {\EuScript L}_2 \}$$ of $(0, 1)$-matrices, regular, sparse such that
\begin{description}
\item[A) ] If $n\geq 4$ be an even integer and $1\leq k\leq r_n-2$,  then  
    $$B_{L(n,E)}={\EuScript L}_{r_n}\oplus \bigoplus_{k=1}^{r_n-2}\Big(\bigoplus_{\substack{1\leq a_1<\cdots<a_{2k}\leq 2n\\ a_i+a_j\neq 2n+1}}{\EuScript L}_{r_n-k}^{(a_1,\cdots, a_{2k})}\Big)$$
    where ${\EuScript L}_{r_n-k}^{(a_1,\cdots, a_{2k})}$ is a copy of ${\EuScript
        L}_{r_n-k}$, for each $1 \leq k \leq r_n-2$.
\item[B)] If  $n\geq 5$ be an odd integer  then  $B_{L(n,E)}$  
     $$B_{L(n,E)}={\EuScript L}_{r_n}^{n}\oplus\bigoplus_{k=1}^{r_n-2}
    \Bigg(\bigoplus_{\substack{1\leq a_1< a_2< \cdots < a_{2k+1}\leq 2n \\ a_i
        + a_j \neq 2n+1}} {\EuScript
        L}_{r_n-k}^{(a_1,a_2,\dots,a_{2k+1})}\Bigg)$$ 
        where ${\EuScript L}_{r_n-k}^{(a_1, a_2,\dots, a_{2k+1})}$ is a copy of ${\EuScript L}_{r_n-k}$
    for each $1\leq k\leq r_n-2$ and ${\EuScript L}_{r_n}^n = {\EuScript L}_{r_n} \oplus \cdots \oplus {\EuScript L}_{r_n}$ $n$-times.
\end{description}
 \end{thm}
\subsection{incidence configurations} 
Let $m\geq 2$ even  integer,  $r_m=\frac{m+2}{2}$ and $\Sigma_m=\{P_1, \ldots, P_m \}$
as in  \ref{setSigma0001}
\begin{equation}\label{Setconf1}
\big( C_{\frac{m}{2}}(\Sigma_m), \;S_{P_{\alpha}} \big)_{P_{\alpha}\in C_{\frac{m-2}{2}}(\Sigma_m)}
\end{equation} 
where $C_{\frac{m}{2}}(\Sigma_m)$ es un $C_{\frac{m}{2}}^{m}$-set and 
\begin{equation}\label{Setconf2}
S_{P_{\alpha}}=\{  P_{\beta}\in C_{\frac{m}{2}}(\Sigma_m): supp\{\alpha\}\subset supp\{ \beta\} \}
\end{equation}
a configuration of subsets of $C_{\frac{m}{2}}(\Sigma_m)$.
\begin{rem}\label{remeq}
Note que $S_{\alpha}=\{P_{\beta}\in C_{\frac{m}{2}}(\Sigma_m) : |supp\{\alpha\}\cap supp\{ \beta\}|=\frac{m-2}{2} \}$
\end{rem}
\subsection{Properties }
Let $m\geq 4$ even integer,  dada la configuracion de incidencia \ref{Setconf1}
$$\big( C_{\frac{m}{2}}(\Sigma_m), \;S_{P_{\alpha}} \big)_{P_{\alpha}\in C_{\frac{m-2}{2}}(\Sigma_m)}$$
 let $r_m=\frac{m+2}{2}$ then we have\\
\begin{lem}\label{onesroow}
 For all $P_{\alpha}\in C_{\frac{m-2}{2}}(\Sigma_m)$ we have $|S_{P_{\alpha}}|=r_m$ 
\end{lem}
\begin{proof}
If $P_{\beta}\in S_{P_{\alpha}}$ and $|supp\{\alpha\}|=\frac{m-2}{2}$ then 
\begin{align*}
|S_{P_{\alpha}}|&=|\{\beta\in I(\frac{m}{2}, m): supp\{\beta \}=supp\{\alpha \}\cup\{ i\}\; with \; 1\leq i \leq m\}|\\
        &=|\{ i\in [m] : |supp\{\alpha\}\cup \{ i \}| =m/2\}|\\
        &=m-|supp\{\alpha\}|\\
        &=m-m/2\\
        &=\frac{m+2}{2}
\end{align*}
\end{proof}
\begin{lem}\label{lem1210}
Let $P_\alpha$ and $P_{\overline{\alpha}}$ two different elements of $C_{\frac{m-2}{2}}(\Sigma_m)$ then 
$$S_{P_{\alpha}}\cap S_{P_{\overline{\alpha}}}\neq \emptyset \;\; if\; and\; only\; if\;\; |supp\{\alpha \}\cap supp \{ \overline{\alpha}\}|=\frac{m-4}{2}$$
\end{lem}
\begin{proof}
$ \Rightarrow):$ Let $ P_{\beta}\in  S_{P_{\alpha}}\cap  S_{ P_{\overline{\alpha}}}$ then there are two different positive integers $N$ and $M$ such that
$supp \{ \alpha \} \cup \{ N\}=supp \{ \beta\} =supp \{ \overline{\alpha}\} \cup \{ M\}$, also 
$supp \{ \alpha\} - \{ M\}=supp \{ \beta\}-\{N, M \} =
supp \{\overline{\alpha}\} - \{ N\}$, as a consequence  we have to $supp \{\beta\}-\{N, M \}= supp \{\alpha\} \cap supp \{ \overline{\alpha}\}$ so
$|supp\{P_{\alpha}\}\cap supp\{ P_{\overline{\alpha}} \} |=|supp \{ P_{\beta}\}|-|\{N, M \}|=\frac{m-4}{2}$\\
$\Leftarrow):$ Suppose $ |supp\{ \alpha\}\cap supp\{\overline{\alpha} \} |=\frac{m-4}{2}$  then there exist $M$, $N$ distinct positive integers such that $\{M\}= supp \{ \alpha\}- supp \{ \alpha\}\cap supp \{\overline{ \alpha}\}$ and  $\{N\}= supp \{\overline{ \alpha}\}- supp \{ \alpha\}\cap supp \{\overline{ \alpha}\}$ and so it exists
 $P_{\beta}\in C_{\frac{m}{2}}(\Sigma_m)$ such that 
$supp \{ \beta\}=(supp \{ \alpha\}\cap supp \{ \overline{\alpha}\})\cup \{ N, M\}$ then it is easy to see 
$P_{\beta}\in S_{ P_{\alpha}}\cap  S_{ P_{\overline{\alpha}}}$ and so $S_{P_{\alpha}}\cap S_{P_{\overline{\alpha}}}\neq \emptyset$.
\end{proof}
\begin{cor}\label{canes1}
Let $P_{\alpha}$ and $P_{\overline{\alpha}}$ be two different of $C_{\frac{m-2}{2}}(\Sigma_n)$ then $|S_{P_{\alpha}}\cap S_{P_{\overline{\alpha}}}|\leq 1$
\end{cor}
\begin{proof}
Be $ P_{\beta}\in S_{ P_{\alpha}}\cap  S_{P_{\overline{\alpha}}}$     
then there are two positive integers $N$ and $M$ different such that
$supp \{ \alpha \} \cup \{ N\}=supp \{ \beta\} =supp \{ \overline{\alpha}\} \cup \{ M\}$, also 
$supp \{ \alpha\} - \{ M\}=supp \{ \beta\}-\{N, M \} =
supp \{\overline{\alpha}\} - \{ N\}$, as a consequence we have to $supp \{\beta\}= (supp \{\alpha\} \cap supp \{ \overline{\alpha}\})\cup \{N, M\} $  clearly 
$(supp \{\alpha \}\cap supp \{ \overline{\alpha}\})\cup \{M\}\subset supp \{\alpha \}$
then for  lemma \ref{lem1210} we have 
$ |(supp\{  \alpha\}\cap supp\{ \overline{\alpha} \})\cup \{M\} |=\frac{m-4}{2}+1=\frac{m-2}{2}$ so  we have to $supp \{\alpha \}=(supp \{\alpha \}\cap supp \{ \overline{\alpha}\})\cup\{M \}$, analogously we have  to $supp \{\overline{\alpha} \}=(supp \{\alpha \}\cap supp \{ \overline{\alpha}\})\cup\{N\}$. Suppose there is another  $P_{\beta^{\prime}}\in S_{P_{\alpha}}\cap  S_{ P_{\overline{\alpha}}}$ then there exist $M^{\prime}$ y $N^{\prime}$ distinct positive integers such that  
$supp\{\beta^{\prime}\}=supp \{\alpha \}\cap supp \{ \overline{\alpha}\}\cup \{N^{\prime}, M^{\prime}\}$, as 
$supp\{\alpha \}\subseteq supp\{ \beta^{\prime}\}$ then $(supp \{\alpha\}\cap supp \{\overline{\alpha}\})\cup \{ M \} \subseteq (supp \{\alpha\}\cap supp \{\overline{\alpha}\})\cup \{  N^{\prime}, M^{\prime} \} $ then $M\in \{ N^{\prime}, M^{\prime}\}$. Analogamente $supp\{\overline{\alpha} \}\subseteq supp\{ \beta^{\prime}\}$ and so we have to $N\in \{ N^{\prime}, M^{\prime}\}$ so $\{N, M\}=\{N^{\prime}, M^{\prime}\}$ then  $supp\{\beta \}=supp\{\beta^{\prime}\}$ 
  so by \ref{indequal000} we have $ |S_{P_{\alpha}}\cap  S_{ P_{\overline{\alpha}}}|\leq 1$
\end{proof}
\begin{cor}\label{canes2}
If $S_{P_{\alpha}} = S_{P_{\overline{\alpha}}}$ then $P_{\alpha}=P_{\overline{\alpha}}$
\end{cor}
\begin{proof}
Suppose $S_{ P_{\alpha}} = S_{P_{\overline{\alpha}}}$ and  
$ P_{\alpha}\neq  P_{\overline{\alpha}}$ then 
 by corollary 
\ref{canes1} we have to $ |S_{ P_{\alpha}}\cap  S_{ P_{\overline{\alpha}}}|\leq 1$ however this is not possible since by lemma \ref{onesroow}, $ |S_{ P_{\alpha}}\cap  S_{ P_{\overline{\alpha}}}|=|S_{ P_{\alpha}}|=r_m$  
but $m\geq 4$ and so $r_m\geq 3$ which implies that $P_{\alpha}=P_{\overline{\alpha}}$
\end{proof}
We call $w_{ P_{\alpha}}$ the intersection count of the $S_{ P_{\alpha}}$  
\begin{lem}\label{unoscolumn}
$w_{P_{\alpha}}=r_m-1$
\end{lem}
\begin{proof}
Clearly  $P_{\beta}\in S_{ P_{\alpha}}$ 
if and only if  $\alpha\in C_{\frac{m-2}{2}}\{supp\{ \beta\} \}$, so the number subsets $S_{ P_{\alpha}}$ that contain  $ P_{\beta}$ is equal to $$|C_{(\frac{m-2}{2})}\{supp\{ \beta\}\}|=C_{(m-2)/2}^{m/2}=m/2=r_m-1$$ 
recuerde que $supp \{ \alpha\} \subseteq supp \{ \beta\}$. 
\end{proof}
Let us denote by 
\begin{equation}\label{Lmatrix}
{\EuScript L}_{r_m}
\end{equation}
the $C_{\frac{m-2}{2}}^m\times C_{\frac{m}{2}}^m$-incidence matrix
from $( C_{\frac{m}{2}}(\Sigma_m), S_{P_{\alpha}} )_{P_{\alpha}\in C_{\frac{m-2}{2}} (\Sigma_m)}$
(\ref{Setconf1})  incidence configuration 
\begin{prop}\label{princmatrx}
Let $2\leq m $ even positive integer and let $r_m=\frac{m+2}{2}$ then
\begin{description}
\item[a)] ${\EuScript L}_{r_m}$ has $r_m$-ones in each row
\item[b)] ${\EuScript L}_{r_m}$ has $r_m-1$-ones in each column
\item[c)] every two lines have at most one $1^{\prime}$ in common
\item[d)] ${\EuScript L}_{r_m}$ is sparse 
\end{description}
\end{prop}
\begin{proof}
The case $m=2$ generates the matrix ${\EuScript L}_2$ which trivially satisfies all the statements of this statement. So we assume that $m\geq 4$ for the rest of the proof.\\
Para  $a)$ the $1's$ in row $P_{\alpha}$ of ${\EuScript L}_{r_m}$ display the elements in the subset $S_{P_{\alpha}}$ so by lemma \ref{onesroow} each row has exactly $r_m$ ones in each row.\\
Para $b)$ the $1's$ in the column $P_{\beta}$ display the occurrences of the elements of $S_{P_{\alpha}}$ among
the subsets, esto se sigue del lema \ref{unoscolumn}.\\
$c)$ follows directly from corollary \ref{canes1}.\\
$d)$ The density of ones in the matrix is given by  $r_m\times C^m_{\frac{m-2}{2}}=(r_m-1)\times C^m_{\frac{m}{2}}$
 $$\frac{r_m}{ C^m_{\frac{m}{2}}}=\frac{r_m-1}{ C^m_{\frac{m-2}{2}}}$$
which approaches zero as $m$ approaches infinity.
\end{proof}
\begin{defn}\label{atlas}
If $m$ and $n$ are even integers such that $2\leq m \leq n$ and let $r_m=\lfloor\frac{m+2}{2}\rfloor$ and $r_n=\lfloor\frac{n+2}{2 }\rfloor$  we define the atlas to the family
\begin{equation}\label{atls}
{\EuScript A}:= \{ {\EuScript L}_{r_n}, {\EuScript L}_{r_n-1}, \ldots, {\EuScript L}_2\}  
\end{equation}
of incidence matrices corresponding to the family
\begin{equation}\label{Setconf111}
\bigg\{\big( C_{\frac{m}{2}}(\Sigma_{m}), \; S_{P_{\alpha}} \big)_{P_{\alpha}\in C_{\frac{m-2}{2}}(\Sigma_{m})}
 \bigg\}_{m=2}^n
\end{equation}
of incidence configurations
\end{defn}
\subsection{Cartesian configurations}
 For $n\geq 4$ even,  $r_n=\frac{n+2}{2}$,  $1 \leq k \leq r_n-2$
and   $1\leq a_1<a_2<\cdots<a_{2k}\leq 2n$ such that $a_i+a_j\neq 2n+1$  then we define 
$\Sigma_{a_1,\ldots,a_{2k}}=\Sigma_n-\{P_{a_1}, \dots, P_{a_{2k}}\}$ as in \ref{setSigma0011}.
We define an cartesian incidence configuration as in \ref{confcart}
  \begin{equation}\label{subsetconf12}
 (a_1,\ldots, a_{2k})\times  \bigg( C_{\frac{n-2k}{2}}(\Sigma_{a_1,\ldots ,a_{2k}}),  \; S_{(a_1,\ldots, a_{2k}, P_{\alpha})}
  \bigg)_{P_{\alpha}\in C_{\frac{n-2(k+1)}{2}}(\Sigma_{a_1,\ldots,a_{2k}})}
  \end{equation}
 where 
  \begin{equation*}\label{subsetconf21}
 \{(a_1,\ldots, a_{2k})\}\times C_{\frac{n-2k}{2}}(\Sigma_{a_1,\ldots ,a_{2k}}) 
\end{equation*}
is an $C^{n-2k}_{\frac{n-2k}{2}}$-set and the subsets are
\begin{equation}\label{subsetconf212}
S_{(a_1,\ldots, a_{2k}, P_{\alpha})}:=\{ (a_1,\ldots, a_{2k}, P_{\beta}) : supp\{\alpha\}\subset supp\{\beta\} \}
 \end{equation} 
 for all $P_{\alpha}\in C_{\frac{n-2(k+1)}{2}}(\Sigma_{a_1,\ldots,a_{2k}})$.
 We denote its $C^{n-2k}_{\frac{n-(2k+1)}{2}}\times C^{n-2k}_\frac{n-2k}{2}$ incidence matrix by ${\EuScript L}_{r_n-k}^{(a_1, \ldots, a_{2k})}$.
\begin{lem}\label{isomconf1}
For $n\geq 4$ even,  $r_n=\frac{n+2}{2}$,  $1 \leq k \leq r_n-2$ and   $1\leq a_1<a_2<\cdots<a_{2k}\leq 2n$ such that $a_i+a_j\neq 2n+1$  then the incidence matrix of \ref{subsetconf12} is ${\EuScript L}_{r_n-k}^{a_1,\ldots,a_{2k}}={\EuScript L}_{r_n-k}$ an element of ${\EuScript A}$ 
\end{lem}
\begin{proof}
As $|\Sigma_{a_1,\ldots,a_{2k}}|=n-2k$ if we make $m=n-2k$ and $r_m=\frac{m+2}{2}$, clearly $2\leq n-2k \leq n-2 $ and $r_m=r_n-k$. A simple calculation shows that  $2\leq r_m \leq r_n-1$, then renumber the elements of $\Sigma_{a_1,\ldots,a_{2k}}= \{ P_1, \ldots, P_m\}$ so by the lemma \ref{incconfisomr} the incidence configuration
$$(a_1,\ldots, a_{2k})\times \bigg(  C_{\frac{n-2k}{2}}(\Sigma_{a_1,\ldots ,a_{2k}}), \;S_{(a_1,\ldots, a_{2k}, P_{\alpha})} \bigg)_{P_{\alpha}\in C_{\frac{n-2(k+1)}{2}}(\Sigma_ {a_1,\ldots,a_{2k}})}$$
 is isomorphic to 
  $$\big( C_{\frac{m}{2}}(\Sigma_{m}), \;\;\; S_{P_{\alpha}} \big)_{P_{\alpha}\in C_{\frac{m -2}{2}}(\Sigma_{m})}$$
  so also for the lemma  \ref{incconfisomr} both  have the same incidence matrix 
  ${\EuScript L}_{r_n-k}^{a_1,\ldots,a_{2k}}={\EuScript L}_{r_n-k}$, as we saw before $2\leq r_m \leq r_n-1$  and so
  ${\EuScript L}_{r_n-k}^{a_1,\ldots,a_{2k}}\in {\EuScript A}$ 
  \end{proof}
For $n\geq 5$ odd and $r_n=(n+1)/2$ and $j \in \{1, \ldots, n \}$, 
we define an cartesian incidence configuration as in \ref{confcart}
\begin{equation} \label{subsetconf122}
j\times \bigg( C_{\frac{n-1}{2}}(\Sigma_n), \; S_{ (j, P_{\alpha})}
 \bigg)_{P_{\alpha}\in C_{\frac{n-3}{2}}(\Sigma_n)}
\end{equation}
where
 \begin{equation*}
\{j\}\times C_{\frac{n-1}{2}}(\Sigma_n) 
\end{equation*} 
is a $C^n_{\frac{n}{2}}$-set and its subsets $S_{(j, P_{\alpha})}$ define by
  \begin{equation}\label{prodcruz0001}
  S_{(j, P_{\alpha})}=\{ (j, P_{\beta}): supp\{\alpha\}\subset supp\{ \beta\}\}\\                         
  \end{equation}
  with $P_{\alpha}\in C_{\frac{n-3}{2}}(\Sigma_n)$. We denote its $C^n_{(n-2)/2\times C^n_{n/2}}$-incidence matrix ${\EuScript L}_{r_n}^j$ for each $ j\in [n]$\\
   For $n\geq 5$ odd and $r_n=\frac{n+1}{2}$,  $1 \leq k \leq r_n-2$, consider
  $1\leq a_1<a_2<\cdots<a_{2k+1}\leq 2n$ such that $a_i+a_j\neq 2n+1$   
  we define
 $\Sigma_{a_1,\ldots,a_{2k+1}}=\Sigma_n-\{P_{a_1}, \ldots, P_{a_{2k+1}} \} $, as in \ref{setSigma0011} \\
We define an cartesian incidence configuration as in \ref{confcart}
  \begin{equation}\label{subsetconf2}
   (a_1,\ldots, a_{2k+1})\times  \bigg ( C_{\frac{n-(2k+1)}{2}}(\Sigma_{a_1,\ldots ,a_{2k+1}}), \; S_{ (a_1,\ldots, a_{2k+1},P_{\alpha}}) \bigg)_{P_{\alpha}\in C_{\frac{n-(2k+3)}{2}}(\Sigma_{a_1\ldots, a_{2k+1}})}
  \end{equation}
where 
  \begin{equation*}
 \{ (a_1,\ldots, a_{2k+1})\} \times C_{\frac{n-(2k+1)}{2}}(\Sigma_{a_1,\ldots ,a_{2k+1}}) 
\end{equation*}
is $C^{n-(2k+1)}_{\frac{n-(2k+1)}{2}}$-set
where the  family of subsets is given by
\begin{equation}\label{prodcruz0001001}
  S_{(a_1,\ldots, a_{2k+1},P_{\alpha})}=\{ (a_1,\ldots, a_{2k+1}, P_{\beta}):
 supp\{\alpha\}\subset supp\{\beta\} \}\\
 \end{equation} 
for all $ P_{\alpha}\in C_{\frac{n-2(k+3)}{2}}(\Sigma_{a_1\ldots, a_{2k}})$.  
${\EuScript L}_{r_m}^{(a_1,\ldots, a_{2k+1})}$ denotes to the $C^{n-(2k+1)}_{\frac{n-(2k+3)}{2}}\times C^{n-(2k+1)}_{\frac{n-(2k+1)}{2}}$ incidence matrix.
\begin{lem}\label{isomconf2}
\begin{description}
\item[a)]
The incidence matrix \ref{subsetconf122} is ${\EuScript L}_{r_n-n}$ 
\item[ b)] 
The incidence matrix \ \ref{subsetconf2} is ${\EuScript L}_{r_m}^{(a_1,\ldots, a_{2k+1})}={\EuScript L}_{r_n-k}$ and is an element of  $ {\EuScript A}$ 
\end{description}
\end{lem}
\begin{proof}
For the proof of $a)$ If we do $m=n-1$ then $n=m+1$ and so on $$j\times \bigg( C_{\frac{n-1}{2}}(\Sigma_n), S_{P_{\alpha}}  \bigg)_{P_{\alpha}\in C_{\frac{n-3}{2}}(\Sigma_n)}\cong \bigg( C_{\frac{m}{2}}(\Sigma_{m+1}), S_{P_{\alpha}} \bigg)_{P_{\alpha}\in C_{\frac{m-2}{2}}(\Sigma_{m+1})} $$ 
so ${\EuScript L}_{r_n}^{\{j \}}={\EuScript L}_{r_m}$ so $r_m=\frac{m+2}{2}=\frac{n+1}{2}=r_n$, then
${\EuScript L}_{r_n}^{\{j \}}={\EuScript L}_{r_n}$  and the part $a)$ has been proved.\\
For the proof of $b)$ clearly $|\Sigma_{a_1,\ldots, a_{2k+1}}|=n-(2k+1)$ if we do $m=(n-1)+2k$ and
$r_m=\frac{m+2}{2}$,  clearly we have that $r_m=r_n-k$. Now
if we rename the elements of $\Sigma_{a_1,\ldots, a_{2k+1}}=\{P_1, \ldots, P_m \}$ then
$$ (a_1,\ldots, a_{2k+1})\times\bigg ( C_{\frac{n-(2k+1)}{2}}(\Sigma_{a_1,\ldots ,a_{2k+1}}), S_{P_{\alpha}} \bigg)_{P_{\alpha}\in C_{\frac{n-(2k+3)}{2}}(\Sigma_{a_1\ldots, a_{2k}})}$$
is isomorphic to the incidence configuration $$( C_{\frac{m}{2}}(\Sigma_{m+1}), S_{P_{\alpha}} )_{P_{\alpha}\in C_{\frac{m-2}{2}}(\Sigma_m)} $$ then ${\EuScript L}_{r_m}^{(a_1,\ldots, a_{2k+1})}={\EuScript L}_{r-k}$ note that  $2\leq m \leq m-3$ implies that $2\leq r_m \leq r_n-1$. 
 \end{proof}
\begin{lem}\label{lem2.5} 
\begin{description}
\item[A) ] 
If $n\geq 4$ and let $r=\lfloor\frac{n+2}{2} \rfloor$,  then   
{ \begin{equation}\label{In2npartc11}
 C_{\frac{n-2}{2}}(\Sigma_{n})\cup\Big(\bigcup_{k=1}^{r-2}
\bigcup_{\substack{1\leq a_1<\cdots < a_{2k}\leq 2n\\ a_i+a_j\neq 2n+1}}(a_1,\ldots, a_{2k})\times C_{\frac{n-2(k+1)}{2}}(\Sigma_{a_1,\ldots ,a_{2k}})\Big).
\end{equation}}
is a partition of the set $I(n-2,2n)$.
\item[B)]
If $n\geq 5$ and let $r=\lfloor\frac{n+1}{2} \rfloor$,  then   
{ \begin{equation}\label{In2npartc12}
\bigcup_{j=1}^n\{j\}\times C_{\frac{n-2}{2}}(\Sigma_{n})\cup\Big(\bigcup_{k=1}^{r-2}
\bigcup_{\substack{1\leq a_1<\cdots < a_{2k}\leq 2n\\ a_i+a_j\neq 2n+1}}(a_1,\ldots, a_{2k})\times C_{\frac{n-2(k+1)}{2}}(\Sigma_{a_1,\ldots ,a_{2k}})\Big).
\end{equation}}
is a partition of the set $I(n-2,2n)$.
\end{description}
\end{lem}
\begin{proof}
Let $n\geq 4$ and let $r_n=\frac{n+2}{2}$,  then   it is sufficient to show that $I(n-2, 2n)$ is contained in \ref{In2npartc11} $($ resp. If $n\geq 5$ and let $r_n=\frac{n+1}{2}$,  then  $I(n-2, 2n)$ is contained in \ref{In2npartc12} $)$. If $supp\{ \alpha_{rs}\}\subseteq \Sigma_n$ then it exists $P_{\theta}\in C_{\frac{n-2}{2}}(\Sigma_n)$ such that
$\alpha_{rs}=P_{\theta}\in C_{\frac{n-2}{2}}(\Sigma_n)$ $($ resp. exists $P_{\theta}\in C_{\frac{n-3}{2}}(\Sigma_n)$ such that
$\alpha_{rs}=(j,P_{\theta})\in \{j \}\times C_{\frac{n-3}{2}}(\Sigma_n)$). 
If $supp\{ \alpha_{rs}\}\cap \Sigma_n= \emptyset$ then $\alpha_{rs}=(a_1, \ldots, a_{n-2})$ such that $a_i+a_j\neq 2n+1$  and so we have to $\alpha_{rs}\in C_0(\Sigma_{a_1, \ldots, a_{n-2}})$  $($ resp. $\alpha_{rs}=(a_1, \ldots, a_{n-3})\in C_0(\Sigma_{a_1, \ldots, a_{n-3}})$ since $a_i+a_j\neq 2n+1$. $)$  
If $supp\{ \alpha_{rs}\}\cap \Sigma_n\neq \emptyset$ then there are  $P_{\theta}\in C_{\frac{n-2k}{2}}(\Sigma_{a_1, \ldots, a_{2k}})$ and $1\leq a_1<\ldots, < a_{2k}\leq 2n$ such that $\alpha_{rs}=(a_1, \ldots, a_{2k}, P_{\theta})\in (a_1, \ldots, a_{2k})\times C_{\frac{n-2k}{2}}(\Sigma_{a_1, \ldots, a_{2k}})$ . $($ resp. exist $P_{\theta}\in C_{\frac{(n-1)-2k}{2}}(\Sigma_{a_1, \ldots, a_{2k+1}})$ and $1\leq a_1<\ldots, < a_{2k+1}\leq 2n$ such that $\alpha_{rs}=(a_1, \ldots, a_{2k+1}, P_{\theta})\in (a_1, \ldots, a_{2k+1})\times C_{\frac{(n-1)-2k}{2}}(\Sigma_{a_1, \ldots, a_{2k+1}})$ $)$
with which the demonstration ends.
\end{proof}
For all $\alpha_{rs}\in I(n-2,2n)$ making a change in the notation we rewrite \ref{FuncLinPI} as 
\begin{gather}\label{FuncLinPI2}
\Pi_{\alpha_{rs}}=\sum_{i=1}^nc_{\alpha_{rs}P_i} X_{\alpha_{rs}P_i} 
\end{gather}
where 
$$c_{\alpha_{rs}P_i}=\begin{cases}
1  & \text{if $|supp\{\alpha_{rs}P_i \}|=n$}, \\
0 & \text{otherwise},
\end{cases}$$
For each $\alpha_{rs}\in I(n, 2n)$ consider
\begin{align*}
S_{\alpha_{rs}}&:=\{ \beta\in I(n, 2n): supp\{ \alpha\}\subset supp\{ \beta \}\}\\
                        &=\{  \beta\in I(n, 2n): |supp\{\alpha\}\cap supp\{ \beta \}|=n-2 \}\\
                        &=\{ \alpha_{rs}P_i \in I(n-2, 2n)\times \Sigma_n : |supp\{\alpha_{rs}P_i \}|=n\} 
\end{align*}
\begin{rem}\label{remsalphars}
Note that depending on where $\alpha_{rs}\in I(n-2, 2n)$ is,  we have  $S_{\alpha_{rs}}$ is equal to
\ref{Setconf2},  \ref{subsetconf212}, \ref{prodcruz0001} or \ref{prodcruz0001001} respectively.  
\end{rem}
\subsection{funcion $\varphi$}
For $n\geq 4$,  we consider
 \begin{gather}\label{extvarphi}
   \varphi:I(n-2,2n)\rightarrow \{0,1\}^{C^{2n}_{n}}\\
   (\alpha_{rs})\mapsto ( c_{\varepsilon} )_{\varepsilon\in I(n,2n)}\ ; where\\
  c_{\beta} = \begin{cases}
1  & \text{if}\; \beta \in S_{\alpha_{rs}}\\
0& \text{otherwise}.
\end{cases}
\end{gather}
Clearly
\begin{equation}
B_{L(n, E)}=\varphi(I(n, 2n))
\end{equation}
up to permutation of rows
\begin{lem}\label{lemmaTri}
The function $\varphi$ is injective
\end{lem}    
\begin{proof}
Let $\alpha_{rs}, \; \alpha^{\prime}_{rs}\in I(n-2, 2n)$ such that $\varphi(\alpha_{rs})=\varphi(\alpha^{\prime}_{rs}) $
then 
$( c_{\varepsilon} )_{\varepsilon\in I(n,2n)}=( c^{\prime}_{\varepsilon} )_{\varepsilon\in I(n,2n)}$  this implies            $S_{\alpha_{rs}}=S_{\alpha^{\prime}_{rs}}$ and by corollary \ref{canes2} we have $\alpha_{rs}=\alpha^{\prime}_{rs}$
\end{proof}
 \begin{cor}\label{bloksmatrx}
 \begin{description}
 Let $n\geq 4$ integer even, $r_n=\frac{n+2}{2}$ then
 \item[a)] $\varphi(C_{\frac{n-2}{2}}(\Sigma_n))={\EuScript L}_{r_n}$
 \item[b)] $\varphi((a_1, \ldots, a_{2k})\times C_{\frac{n-2k}{2}}(\Sigma_{a_1,\ldots ,a_{2k}}))={\EuScript L}^{(a_1, \ldots, a_{2k})}_{r_n-k}$ \\
 Let $n\geq 5$ odd integer, $r_n=\frac{n+1}{2}$ then
  \item[c)] $\varphi(\{ j\}\times C_{\frac{n-2}{2}}(\Sigma_n))={\EuScript L}^j_{r_n}$ for all $j=1, \ldots, n$
 \item[d)] $\varphi((a_1, \ldots, a_{2k+1})\times C_{\frac{n-2(k+1)}{2}}(\Sigma_{a_1,\ldots ,a_{2k+1}}))={\EuScript L}^{(a_1, \ldots, a_{2k+1})}_{r_n-k}$ 
 \end{description}
 \end{cor}
 \begin{proof}
 The proof follows directly from the lemma \ref{lemmaTri} and remark \ref{remsalphars}.
 \end{proof} 
\subsection{Proof of the theorem \ref{resinic000}}   
   $A)$: 
    by Lemma \ref{lem2.5}  we\; have 
\begin{align*}
        I(n-2,2n)&=
        C_{\frac{n-2}{2}}(\Sigma_{n})\cup\Big(\bigcup_{k=1}^{r_n-2}
\bigcup_{\substack{1\leq a_1<\cdots < a_{2k}\leq 2n\\ a_i+a_j\neq 2n+1}}(a_1,\ldots, a_{2k})\times C_{\frac{n-2k}{2}}(\Sigma_{a_1,\ldots ,a_{2k}})\Big).
\end{align*}
Since $\varphi$ is injective we have a partition in the image, 
 \begin{align*}
        {\varphi(I(n-2, 2n))} &= \varphi(C_{\frac{n-2}{2}}(\Sigma_{n})) \bigcup
        \Bigg(  \bigcup_{k-1}^{r_n-2} \bigcup_{\substack{1\leq a_1<\cdots < a_{2k}\leq 2n \\ a_i+a_j\neq 2n+1}} \varphi(_{(a_1,\ldots, a_{2k})\times C_{\frac{n-2(k-1)}{2}}(\Sigma_{a_1,\ldots ,a_{2k}})})\Bigg)
  \end{align*}
Associating the corresponding matrix, using the corollary \ref{bloksmatrx}, we have that 
  \begin{align*}      
        {B_{L(n,E)}} & = {\EuScript L}_{r_n}\oplus \bigoplus_{k=1}^{r_n-2}\Big(\bigoplus_{\substack{1\leq a_1<\cdots<a_{2k}\leq 2n\\ a_i+a_j\neq 2n+1}}\!\!\!\!{\EuScript L}_{r_n-k}^{(a_1,\cdots, a_{2k})}\Big)
\end{align*} 
\qed\\
Part $B)$  
Proceeding as in part $A)$ of this proof, we have    
\begin{multline*}    
      \bigcup_{j=1}^n(\{j\}\times C_{\frac{n}{2}}(\Sigma_{n}))\cup\Big(\bigcup_{k=1}^{r_n-2}
\bigcup_{\substack{1\leq a_1<\cdots < a_{2k+1}\leq 2n\\ a_i+a_j\neq 2n+1}}(a_1,\ldots, a_{2k+1})\times C_{\frac{n-2k}{2}}(\Sigma_{a_1,\ldots ,a_{2k+1}})\Big)
\end{multline*} 
 we obtain
\begin{align*} \varphi_{(I(n-2,2n))}&=\\
      & =  \bigcup_{j=1}^n \varphi(\{j\}\times C_{\frac{n}{2}}(\Sigma_{n})) \cup  \bigcup_{k=1}^{r-2}\Bigg ( \bigcup_{\substack{1\leq a_1<
            a_2< \ldots < a_{2k+1}\leq 2n \\ a_i + a_j \neq 2n +1}}
        \varphi_{(\Sigma( a_1, a_2,\ldots, a_{2k+1})})\Bigg)   \\
        B_{L(n,E)}&=\bigoplus_{j=1}^n {\EuScript L}_{r_n}^j \oplus
        \bigoplus_{k=1}^{r_n-2} \Bigg(\bigoplus_{\substack{1\leq a_1 < \ldots <
            a_{2k+1} \leq a_{2n} \\ a_i + a_j \neq 2n+1}}  {\EuScript
            L}_{r_n-k}^{(a_1, a_2,\ldots, a_{2k+1})}\Bigg)          
 \end{align*}
 \qed\\
 %
\section{Consequences}
\begin{prop}\label{UnicoM}
$\ker f$ is the smallest linear subspace which contains $L(n, E)$
\end{prop}
\begin{proof}
Suppose there exists $W$ vector subspace of ${\mathbb P}(\wedge ^{n}E)$ proper such that $L (n, E) \subseteq W\subset {\mathbb P}(\ker f)$ then
exists $H$ hyperplane of ${\mathbb P}(\wedge ^{n}E)$ of codimension 1 such that $W \subset H$ then exists 
$h \in (\bigwedge ^{n}E)^{*}$  such that $H=Z\langle h \rangle$ and $\{ \Pi_{\alpha_{rs}}, h: \alpha_{rs}\in I(n-2, 2n)\}$ are linearly independent, however since they also $L(n, E) \subset H$ then $h(L(n, E))=0$ so by the theorem \ref{UniP2} we have $h \in \langle \Pi_{\alpha_{rs}} :  \alpha_{rs}\in I(n-2, 2n)\rangle$ 
which is absurd.
\end{proof}
\begin{lem}\label{contoisom01}
Let $g:\wedge^nE\longrightarrow \wedge^{n-2}E$ a linear transformation such that $g(L(n, E))=0$
then $g(\ker f)=0$
\end{lem}
\begin{proof}
$g(L(n, E))=0$ we have $L(n, E)\subseteq \ker g$ and by proposition \ref{UnicoM} we have to $\ker f\subseteq \ker g$
 that is $g(\ker f)=0$.
\end{proof}

We say that the {\textit embedding rank} $er(L(n, E))$ of $L(n, E)$ is the dimension of the smallest projective space  which contains $L(n, E)$, under the Pl\"ucker embedding
\begin{lem}
Embedding rank of $L(n, E)$ is $er(L(n, E))=C^{2n}_{n}-rank B_{L(n, E)}$ 
\end{lem}
\begin{proof}
  Por el proposition \ref{UnicoM} tenemos que  $er(L(n, E))=\dim ker f$,  then we have
  $re(L(n, E))=C^{2n}_{n}-rank B_{L(n, E)}$.
\end{proof}
\begin{cor}\label{FuntContr}
Let $E$ a symplectic vector space of dimension $2n$ defined in a fields ${\mathbb F}$, let $f$ the contraction map, $r_n=\lfloor \frac{n+2}{2}\rfloor$   are equivalent
\begin{description}
\item[ 1)]  $f$ is surjective
\item [2)] $er(L(n, E))=C_{n}^{2n}-C_{n-2}^{2n}$
\item [3)] $rank B_{L(n, E)}=C_{n-2}^{2n}$ 
\item[4)] char${\mathbb F}=0$ or char${\mathbb F}\geq r_n$
\item[5)] $rank {\EuScript L}_{r-k}$ is maximum for everything $0\leq k \leq r_n-2 $
\end{description}
\end{cor}
\begin{proof}
$1)$ is equivalent $2)$, $2)$ is equivalent $3)$ given that $\dim \ker f= \dim \ker B$ y $\dim \ker f= \binom{2n}{n}-rank B$.  \\
$3)$ is equivalent $4)$ is followed from \cite[theorem 6]{bib2.1} and finally $3)$ is equivalent $5)$ is obvious.
\end{proof}
In \cite[corollary 1.2]{bib 2.122} can find a more general case of corollary \ref{FuntContr}  numeral $4)$
also see  \cite{bib 2.1223} for some examples of the $er(L(n, E))$ for $n=2,3,4,5, 6$ and $7$.\\
Let $E$ be a simplectic vector space of dimension $2n$ and let $r_n=\lfloor\frac{n+2}{2}\rfloor$ consider the family of matrices given in \ref{atls} 
\begin{equation}\label{e-atlas}
{\EuScript A}=\{ {\EuScript L}_{r_n}, {\EuScript L}_{r_n-1}, \ldots, {\EuScript L}_2 \}
\end{equation}
and we'll call it the $r_n$-atlas.
\begin{lem}\label{sameAtlas1}
Let $n$ even integer and let $E$ and $\overline{E}$ symplectic vector spaces of dimension $2n$ and $2m$ respectively then
$$r_n=r_m\; if\; and\; only\; if \;n=m\; or\;m=n+1$$ where
 $r_n=\frac{n+2}{2}$ and $r_m=\frac{m+2}{2}$
 \end{lem}
\begin{proof}
$\Rightarrow)$ If both $m$ and $n$ are even integers or if both $m$ and $n$ odd integers then $n=m$\\
Now suppose that  $n$ is even integer and $m=2k+1$ is odd integer.  If  $r_n=r_m$ then $\lfloor \frac{n+2}{2}\rfloor=\lfloor \frac{2k+1+2}{2}\rfloor$ so $\frac{n+2}{2}=\frac{2k+2}{2}$ what it implies $n=2k$ so $m=n+1$.\\
$\Leftarrow)$ If $n=m$ then $r_n=r_m$. Now if $m=n+1$ then $r_m=\lfloor\frac{m+2}{2} \rfloor=
\lfloor\frac{n+1+2}{2} \rfloor=\frac{n+2}{2}=r_n$
\end{proof}
\begin{cor}\label{sameAtlas2}
$a)$ If $E$ and $\overline{E}$ both are symplectic vector spaces of dimension $2n$ so they share the same $r_n$-atlas.\\
$b)$ Let $E$ symplectic vector space of dimension $2n$ and let $\overline{E}$ symplectic vector space of dimension $2(n+1)$ then both spaces share the same $r_n$-atlas.
\end{cor}
\begin{proof}
The proof follows directly from the lemma \ref{sameAtlas1}.
\end{proof}
\begin{cor}\label{contoisom}
Suppose the contraction map $f$ is surjective and suppose $G:\wedge^nE\longrightarrow \wedge^{n-2}E$ is a surjective linear transformation that vanishes  $L(n, E)$ then there exists a unique isomorphism such that $G=h\circ f$ 
\end{cor}
\begin{proof}
By lemma \ref{contoisom01} we have  $\ker f\subseteq \ker G$, more over $\ker f= \ker G$ since both have the same dimension because $f$ and $G$ are surjective,  then there exists a unique linear isomorphism $h$ that makes the following diagram commute.
$$ \xymatrix{
      0 \ar[r] &\ker f \ar [r] \ar [d]^{Id}   &\wedge^nE \ar [r]^f \ar [d]^{ Id} & \wedge^{n-2}E \ar   [d]^{\exists !h}  \ar[r] & 0\\
      0 \ar[r] &\ker G \ar [r]    &\wedge^nE \ar  [r]^G & \wedge^{n-2}E   \ar[r] & 0 }$$
      and so we have to $G=h\circ f$
\end{proof}
%
Following \cite[section 3]{bib2} we have the following lemma

\begin{lem}\label{contr123}
Let $E$ symplectic vector space and $f$ the contraction map, if
$w=\sum_{\alpha\in I(n,2n)}R_{\alpha}e_{\alpha}\in\wedge^nE$ arbitrary element, in coordinates of pl\"ucker,  then
the contraction map
 $$f(w)=\sum_{\alpha_{rs}\in I(n-2, 2n)}K_{\alpha_{rs}}e_{\alpha_{rs}}$$
 where $K_{\alpha_{rs}}=\sum_{i=1}^{n}R_{(\alpha_{rs},P_i)}$ and $R_{(\alpha_{rs},P_i)}\neq 0$ if $|supp\{ \alpha_{rs}P_i \}|=n$
\end{lem}
\begin{proof} Let $w=\sum_{\alpha\in I(n,2n)}P_{\alpha}e_{\alpha}\in\wedge^nE$ arbitrary elemen
 then
\begin{align*}
f(w)=&\sum_{\alpha\in I(n,2n)}R_{\alpha}f(e_{\alpha})\\
=&\sum_{\alpha\in I(n,2n)}R_{\alpha}\big(\sum_{1\leq r< s \leq n}\langle e_{\alpha_r}, e_{\alpha_s}\rangle \big(-1\big)^{r+s-1}e_{\alpha_{rs}}\big)\\
=&\sum_{1\leq r< s \leq n}\big(\sum_{i=1}^{n}R_{(\alpha_{rs}, P_i)}\langle e_i, e_{2n-i+1}\rangle \big(-1\big)^{[i+2n-i+1]-1} \big)e_{\alpha_{rs}}\\
=&\sum_{\alpha_{rs}\in I(n-2, 2n)}\big(\sum_{i=1}^{n}R_{(\alpha_{rs}, P_i)}\big)e_{\alpha_{rs}}
\end{align*}
\end {proof}
\begin{lem} \label{ConmtSq}
Let $E$ be a symplectic vector space with symplectic basis $\{e_i \}_{i=1}^{2n}$ let $f$ be the contraction map, $\phi$ pl\"ucker embedding  so the following diagram commutes
$$ \xymatrix{\wedge^nE \ar [r]^f \ar [d]^{\phi} & \wedge^{n-2}E \ar [d]^{\phi} \\
                      P(\wedge^nE) \ar  [r]^{B_{L(n, E)}} & P(\wedge^{n-2}E)  } $$
\end{lem}
\begin{proof}
If $w=\sum_{\alpha\in I(n,2n)}R_{\alpha}e_{\alpha}\in\wedge^nE$ and of \ref{Bmatrixb0} then $\phi\big(w \big)=[R_{\alpha}]_{\alpha\in I(n, 2n)}$ and by a direct calculation we have
 $$B_{L(n,E)}(\phi(w))=B_{L(n,E)}[R_{\alpha}]^T_{\alpha\in I(n, 2n)}=\big[\sum_{i=1}^{n}R_{(\alpha_{rs}, P_i)}\big]_{\alpha_{rs}\in I(n-2, 2n)}$$
where $[R_{\alpha}]^T_{\alpha\in I(n, 2n)}$ denotes the transposed vector and by the lemma \ref{contr123} we have  
$$\phi(f(w))=\big[\sum_{i=1}^nR_{(\alpha_{rs}, P_i)}\big]_{\alpha_{rs}\in I(n-2, 2n)}$$ 
which proves commutativity
\end{proof}
%
\begin{cor}
If $f$ is surjective then $ \ker f \cong \ker B_{L(n, E)}$ is a linear isomorphism.
\end{cor} 
\begin{proof}
Clearly $\phi(\ker f)\subset \ker B_{L(n, E)}$, and both have the same dimension so
 $ \ker f \cong \ker B_{L(n, E)}$ is a linear isomorphism.
\end{proof}
\begin{cor}\label{sse}
 Suppose the contraction map $f$ is surjective then
\begin{enumerate}
\item[i)] If $H$ is a matrix of order $C_{n-2}^{2n} \times C_n^{2n}$ and maximum rank that annuls the rational points of 
$L(n,E)(\mathbb{F}_q)$, then $H= PB_{L(n, E)}$, where $P$ is an invertible matrix.

\item[ii)] Suppose that there exists $R$ matrix such that $L(n,E) = G(n,E) \cap \ker R$. Then $R=PB_{L(n, E)}$ where $P$ is an invertible matrix.
\end{enumerate}    
 \end{cor}
\begin{proof}
The proof of $i)$ follows directly from the lemma \ref{contoisom}.
 For the $ii)$ suppose that  
 $R=\left(\begin{smallmatrix}
h_1\\
 h_2 \\ 
 \vdots\\
 h_{\epsilon}
        \end{smallmatrix}\right)$
 is a rank matrix $\epsilon$  such that $L(n, E)=G(n, E)\cap \ker R$, then $L(n, E)\subset \ker R$ and $\epsilon \leq C^n_{n-2}$. If $\epsilon=C^n_{n-2}$ the affirmation is followed by the previous clause of this lemma.
Now suppose that $t<C^n_{n-2}$ then $\ker B_{L(n, E)} \varsubsetneq \ker R$ this implies that
\begin{align*}
L(n, E)&=G(n, E)\cap \ker B_{L(n, E)}\\
           &\varsubsetneq G(n, E) \cap \ker R \\
           &= L(n, E)
 \end{align*}
 which is a contradiction and therefore $\epsilon =C^n_{n-2}$.
\end{proof}
\begin{prop}\label{IdRadLag}
If the field of definition of the symplectic vector space $E$ is algebraically closed then 
 $$I(L(n,E))=\sqrt{\big\langle  Q_{\alpha^{\prime}, \beta^{\prime}}, \Pi_{\alpha_{rs}} \big\rangle}$$ 
so $L(n, E)$ is a projective variety.
\end{prop}
\begin{proof}
By the formula \ref{lagrzeroset} we have $L(n,E)=Z\langle Q_{\alpha^{\prime}, \beta^{\prime}}, \Pi_{\alpha_{rs}}: \alpha^{\prime}\in I(n-1, 2n),  \beta^{\prime}\in I(n+1, 2n),  \alpha_{rs}\in I(n-2,2n)\rangle$ and by  Hilbert$^{\prime}$s Nullstellensatz theorem, ver \cite[Theorem 1.3 A]{bib 2.1233}, so the result is fulfilled.
\end{proof}
We define the ideal $I_{\overline {\mathbb F}_q}\subset{\mathbb F}[x_{\alpha}]_{\alpha\in I(n,2n)}$ as 
\begin{equation}\label{IdealRP}
I_{\overline {\mathbb F}_q}=\big\langle Q_{\alpha^{\prime}, \beta^{\prime}}, \Pi_{\alpha_{rs}}, g_{\alpha} \big\rangle
\end{equation}
where $ \alpha^{\prime}\in I(n-1, 2n),  \beta^{\prime}\in I(n+1, 2n),  \alpha_{rs}\in I(n-2,2n), \alpha\in I(n, 2n)$ y  $g_{\alpha}=x^q_{\alpha}-x_{\alpha}$
\begin{lem}\label{IdRad}
The ideal $I_{\overline {\mathbb F}_q}$ is radical
\end{lem}
\begin{proof}
The ideal  $I_{\overline {\mathbb F}_q}$is zero-dimensional since the set of solutions to the homogeneous polynomial equations $$ |Sol\big\{ Q_{\alpha^{\prime}, \beta^{\prime}}, \Pi_{\alpha_{rs}}, g_{\alpha}\}|<\infty$$ given that  \ref{rationPoint1} implies $|Sol\big\{ Q_{\alpha^{\prime}, \beta^{\prime}}, \Pi_{\alpha_{rs}}, g_{\alpha}\}|=|L(n,E)({\mathbb F}_q)|$ moreover   $g_{\alpha}=qx_{\alpha}^{q-1}-1=-1$ 
so $gcd(g_{\alpha}, g_{\alpha}^{\prime})=1$, and by Seindeber's  lemma, ver  \cite[proposition 3.7.15]{bib 2.2},  $I_{\overline {\mathbb F}_q}$ is radical.  
\end{proof}
%
Let $H_r\leq {\mathbb P}(\wedge^nE)$ a hyperplane of codimension $r$ and suppose that
$H_r=\langle h_1, h_2, \ldots, h_r \rangle$ with $\{h_1, h_2, \ldots, h_r \}\subseteq (\wedge^nE)^*$, 
let $I_r=\langle Q_{\alpha^{'}, \beta^{'}}, \Pi_{rs}, x_{\alpha}^q-x_{\alpha}, h_1, h_2, \ldots, h_r\rangle$.
By $H_r\leq {\mathbb P}(\wedge^nE)$ a hyperplane of codimension $r$, we say that $L(n,E)\cap H_r$ is a {\it linear section of the Lagrangian-Grassmannian}
\begin{lem}
Suppose the basis field ${\mathbb F}_q$ is perfect then the linear section of the Lagrangian-Grassmannian satisfies $|L(n, E)({\mathbb F}_q)\cap H_r|=\dim_{{\mathbb F}_q} {\mathbb F}[x_{\alpha}]_{\alpha\in I(n, 2n)}/I_r$
\end{lem} 
\begin{proof}
Given the $|L(n, E)({\mathbb F}_q)\cap H_r|\leq |L(n, E)({\mathbb F}_q)|$ then the ideal $I_r$ is zero dimensional, $mcd (g_{\alpha}, g_{\alpha}^{\prime})=1$, thus by Seindeber's lemma we have that the ideal $I_r$ is radical.
\end{proof}
From proposition 3.9 of \cite{bib0.0001}, we have that the Lagrangian-Grassmannian manifold $L(n,E)=G(n,E)\cap{\mathbb P}(\ker f)$ is osculating well-behaved and the tangent space
\begin{equation}\label{tangSpace}
T(L(n,E))= T(G(n,E))\cap {\mathbb P}(ker f)
\end{equation}
\begin{cor}
The tangent space for the Lagrangian-Grassmannian manifold is given by $$T(L(n,E))=Z\big\langle \frac{\partial Q_{\alpha^{'}, \beta^{'}}}{\partial x_{\alpha}}: \alpha \in I(n, 2n)\big \rangle\cap Z\big\langle \Pi_{\alpha_{rs}} : \alpha_{rs}\in I(n-2, 2n)\big\rangle$$
\end{cor}
\begin{proof}
It follows directly from \ref{tangSpace}.
\end{proof}
%


\end{document}